\documentclass[a4paper,11pt,reqno]{amsart}

\usepackage[utf8]{inputenc}
\usepackage[english]{babel}
\usepackage{amsmath}
\usepackage{amsthm}
\usepackage{amsfonts}
\usepackage{amssymb}
\usepackage{caption}
\usepackage{cite}
\usepackage{hyperref} 
\usepackage[english]{cleveref} 
\usepackage{emptypage}
\usepackage[shortlabels]{enumitem}
\usepackage{extpfeil}
\usepackage{faktor}
\usepackage{fancyhdr}
\usepackage{float}
\usepackage{graphicx}
\usepackage{indentfirst}
\usepackage{mathtools}
\usepackage{mathabx}
\usepackage{mathrsfs}
\usepackage{multirow}
\usepackage{rotating}
\usepackage{stackrel}
\usepackage{stmaryrd}
\usepackage{subcaption}
\usepackage{tikz}
\usepackage{verbatim}

\usetikzlibrary{matrix,arrows,decorations.pathreplacing}
\newcommand*\Z{\mathbb{Z}}
\newcommand*\N{\mathbb{N}}
\newcommand*\Q{\mathbb{Q}}
\newcommand*\C{\mathbb{C}}
\newcommand*\A{\mathcal{A}}
\newcommand*\B{\mathcal{B}}
\newcommand*\G{\mathcal{G}}

\renewcommand*\S{\mathscr{S}}

\newtheorem{teo}{Theorem}[section]

\newtheorem{prop}[teo]{Proposition}
\newtheorem{cor}[teo]{Corollary}
\newtheorem{lemma}[teo]{Lemma}
\theoremstyle{definition}
\newtheorem{de}[teo]{Definition}
\newtheorem{example}[teo]{Example}
\theoremstyle{remark}
\newtheorem{remark}[teo]{Remark}

\DeclareMathOperator{\Aut}{Aut}

\DeclareMathOperator{\Ext}{Ext}
\DeclareMathOperator{\gr}{gr}
\DeclareMathOperator{\gl}{GL}
\DeclareMathOperator{\Hom}{Hom}
\DeclareMathOperator{\id}{Id}
\DeclareMathOperator{\im}{Im}
\DeclareMathOperator{\M}{M}

\DeclareMathOperator{\pr}{pr}
\DeclareMathOperator{\rad}{Rad}
\DeclareMathOperator{\rk}{rk}
\DeclareMathOperator{\spec}{Spec}

\DeclareMathOperator{\sgn}{sgn}
\DeclareMathOperator{\Tor}{Tor}

\DeclareMathOperator{\T}{T}
\DeclareMathOperator{\Rad}{Rad}

\newcommand{\ie}{i.e.\ }
\newcommand{\eg}{e.g.\ }

\newcommand{\defeq}{\stackrel{\textnormal{def}}{=}}

\begin{document}
\title[Combinatorics of Toric Arrangements]{Combinatorics of Toric Arrangements}
\author[R. Pagaria]{Roberto Pagaria}
\address{Scuola Normale Superiore\\ Piazza dei Cavalieri 7, 56126 Pisa\\ Italia}\email{roberto.pagaria@sns.it}

\begin{abstract}
In this paper we build an Orlik-Solomon model for the canonical gradation of the cohomology algebra with integer coefficients of the complement of a toric arrangement.
We give some results on the uniqueness of the representation of arithmetic matroids, in order to discuss 
how the Orlik-Solomon model depends on the poset of layers.
The analysis of discriminantal toric arrangements permits us to isolate certain conditions under which two toric arrangements have diffeomorphic complements.
We also give combinatorial conditions determining whether the cohomology algebra is generated in degree one.
\end{abstract}
\maketitle

\section*{Introduction}
The aim of this work is to study the combinatorics and the ring structure of the cohomology algebra of toric arrangements.
A toric arrangement is a collection of a finite number of $1$-codimensional subtori in an algebraic torus $(\C^*)^r$. 
Toric arrangements are a natural generalization of hyperplane arrangements. 
Hyperplane and toric arrangements are particular cases of abelian arrangements, but very little is known in this generality: the principal results are about models of the cohomology of the complements (in \cite{MPP91,Bibby,Dupa}).

Let us briefly illustrate the state of the art of toric arrangements.
In 1993 Looijenga, in his article \cite{Looijenga}, used sheaf theory to compute the Betti numbers of the complement of toric arrangements.
Later in \cite{DeConcini} De Concini and Procesi gave a presentation of the cohomology module with complex coefficients using algebraic techniques and pointed out the relation between toric arrangements, partition functions and box splines.
The relation just mentioned is deeply examined in their book \cite{Procesi}, in which they emphasize the connection between hyperplanes and toric arrangements.
The combinatorial aspect and the relation with arithmetic matroids were investigated in \cite{MociTutte2012,BM14,DADDERIO, Lenz}.
The study of homotopy type of toric arrangements appeared for the first time in \cite{Settepanella} and was analyzed in \cite{dAntonio,dAD2012} where the toric Salvetti's complex is studied and the minimality and torsion-freeness of integer cohomology is shown.
A deeper inspection of the Salvetti's complex was made in \cite{Callegaro} (see also the erratum); this leads to a presentation of the cohomology module with integer coefficients.
The wonderful model was described in \cite{Moci12,wonderfulmodel,DCGII} and the formality was proven in \cite{Dupb}.

A purely combinatorial presentation of the ring structure of the cohomology has been given only in few particular cases in \cite{DeConcini} (for totally unimodular arrangements) and in \cite{Lenz} (for weakly multiplicative centred arrangements).
In this article we give a presentation of a gradation of the cohomology ring with rational and integer coefficients.
In the case of rational coefficients the presentation (\cref{thm:main_on_rational}) depends only on the poset of layers (connected components of intersections of subtori).
Dealing with integers is more subtle; we exhibit two toric arrangements with the same poset of layers and different cohomology algebra (\cref{controesempio}).
In fact the presentation that we give (\cref{main_on_integers}) depends on the poset of layers and on another algebraic invariant; these two data permit to reconstruct the characters associated to the hypertori of the arrangements.
The case of centred arrangements is very strictly correlated to arithmetic matroids.
We obtain a uniqueness result about representations of arithmetic matroids (see \cref{teo_uniq_repr}).

The outline of this paper is the following:
after an introduction to fix the notations, we study the Leray spectral sequence of the inclusion of the toric arrangement in the torus.
The spectral sequence for the constant sheaf with integer values collapses at the second page and converges to a gradation of the cohomology ring of the toric arrangement.
In the end of \cref{sect_Leray_ss} we construct a presentation of a bigraded algebra isomorphic to the gradation of the cohomology with integer coefficients.

In \cref{sect_On uniqueness of representability}, we analyze the information codified in the poset of layers and in the arithmetic matroid.
The information codified in the multiplicity function permits to determine the linear relations between the characters.
In the case of surjective arrangements, this implies that the arithmetic matroid has a unique representation or, in other words, the poset of layers determines the arrangement up to torus automorphism.

Combining the results from \cref{sect_Leray_ss} and \ref{sect_On uniqueness of representability}, in \cref{sect_OS_algebra} we can give a completely combinatorial presentation of the gradation of the cohomology ring with rational coefficients.

\Cref{sect_coverings of tori} is devoted to analyze coverings of tori, in order to give a presentation with integer coefficients.
The main result of this section is the definition of coherent element with respect to a poset of layers; the set of coherent elements is in bijection with the set of centred toric arrangements with a fixed poset, up to automorphism of the torus.

In \cref{sect_cohom_int_coeff} we develop the analogous results of \cref{sect_OS_algebra} in the case of integer coefficients.
The information needed to describe the graded ring are codified in the poset of layers together with the coherent element studied in the previous section.

In \cref{sect_example}, we give an example of two toric arrangements with the same poset of connected components of intersections but with different cohomology ring on integers.
This shows that the coherent element associated to the arrangement is necessary to describe the cohomology ring on integers.

This study continues in \cref{sect_discr} with the introduction of discriminantal toric arrangements in a, possibly disconnected, torus.
We give the definition of poset isotopy (roughly speaking it consists in translating the hypertori without changing the poset of layers) and we prove that poset isotopy equivalent arrangements are diffeomorphic.
Moreover we show that not all arrangements with the same characters and poset of layers are poset isotopy equivalent: this property depends on which connected component of the discriminantal arrangement they belong to.

Finally in \cref{sect_gen_deg_one}, we give a purely combinatorial criteria to determine whether the cohomology ring (with rational or integer coefficients) is generated in degree one.

\tableofcontents

\section{Notations}
\label{sect_notation}
Let $T\simeq(\C^*)^r$ be a complex torus of dimension $r$ and $\Lambda$ its character group.
We call \textit{affine hypertorus} the zero locus of $1-a\chi$ where $a\in \C^*$ is a complex number and $\chi \in \Lambda$ is a character.
A \textit{toric arrangement} $\Delta$ is a finite collection of hypertori, corresponding to the data $\Delta=\{(a_e, \chi_e)\}_{e \in E}$, where $E$ is a finite set.
Sometimes, with an abuse of notation, we will write $\chi_e \in \Delta$.
We want to study the complement of the arrangement, defined as $M(\Delta) =T\setminus \bigcup_{e \in E} \mathcal{V}_T (1-a_e \chi_e)$ and determine the ring structure of the cohomology of $M$ (we sometimes omit the dependence on $\Delta$).
A character is \textit{primitive} if it is not a nontrivial multiple of an element in $\Lambda$; a toric arrangement $\Delta$ is \textit{primitive} if all its characters are primitive.
It is always possible to describe the open subset $M(\Delta)$ as the complement of a primitive arrangement $\Delta'$ (\ie $M(\Delta)=M(\Delta')$).
A toric arrangement is \textit{centred} if $a_e=1$ for all $e \in E$.

The main combinatorial object used to study $M$ is the poset $\mathscr{S}$ of connected components of intersections.
Elements in $\mathscr{S}$ are called \textit{layers}, so  the poset $\S$ is often called poset of layers.
Layers are ordered by reverse inclusion, so that $\mathscr{S}$ is a ranked poset with rank function given by the codimension in the torus $T$; we call $\mathscr{S}_k$ the subset of $\mathscr{S}$ of codimension $k$ layers.

We use the Greek letter $\Gamma$ to denote the sublattice of $\Lambda$ generated by the characters $\chi_e$ in the arrangement $\Delta$.
If $\Gamma$ has maximal rank in $\Lambda$ then we say that the arrangement is \textit{essential}.
The study of toric arrangements can be easily reduced to the study of essential arrangements.
For an essential arrangement, the quotient $\Lambda / \Gamma$ is denoted $G$.

A nice property of toric arrangements is that at every point of the torus $T$, the complement $M(\Delta)$ locally looks like the complement of a hyperplane arrangement; for this reason we will use the well-known theory of hyperplane arrangements.
For a general and complete reference, see \cite{orlik-terao}.
In fact, for every point $x$ of the torus $T$, we can define a hyperplane arrangement, associated to $\Delta$, in the tangent space at $x$:
every hypertorus $H$ passing through $x$ has as tangent space a hyperplane in $\T _x T$, that we denote by $H^\upharpoonright$.
We define the hyperplane arrangement
\[ \Delta[x] \defeq \{ H^\upharpoonright \mid H \in \Delta, \, x\in H \} \]

Let $W$ be a layer, we choose a point $x \in W$ and we define the hyperplane arrangement in $\T _x T$ as
\[ \Delta[W] \defeq \{ H^\upharpoonright \mid H \in \Delta, \, H \supseteq W\} \]
The choice of another point in $W$ gives a canonically isomorphic hyperplane arrangement, so the formula above is well defined.
Notice that $\Delta[x]=\Delta[W]$ only for generic $x \in W$, \ie for $x$ in a open dense subset of $W$.

In the setting of hyperplane arrangements there is a natural map called \textit{Brieskorn inclusion}, defined as follows.
Fix a layer $L$ in a hyperplane arrangement $\A$ with poset of intersection $\mathscr{L}$ and let $\A_L$ be the arrangement given by hyperplanes containing $L$.
The complement of the hyperplane arrangement $\A$ will be denoted by $M(\A)$.
The Brieskorn inclusion is the composition
\[ H^k(M(\A_L);\Z) \hookrightarrow \bigoplus_{L\in \mathscr{L}_k } H^k(M(\A_L); \Z) \xrightarrow{\sim} H^k(M(\A);\Z) \]
where the second map is the \textit{Brieskorn isomorphism} (see \cite[Theorem 3.26, p. 65]{orlik-terao} or \cite[Lemma 3, p.27]{Brieskorn}).

\section{The Leray spectral sequence}
\label{sect_Leray_ss}
In this section we state some general results on the Leray spectral sequence, see \cite{Bredon} for a reference.
The case of cohomology with rational coefficients has been studied by Bibby in \cite{Bibby}.
We make use a result appeared for the first time in \cite{Looijenga} to compute the cohomology with integer coefficients of a toric arrangement.
Using the Leray spectral sequence we obtain a nice presentation of a canonical bi-gradation of cohomology algebra of toric arrangements.

Let $j \colon M \hookrightarrow T$ be the natural inclusion, which is a continuous map between topological spaces.
Let $\Z_M$ be the sheaf on $M$ of locally constant functions with values in $\Z$.

We recall the definition of \textit{higher direct image sheaves} for the map $j$ and the sheaf $\Z_M$.
Let us consider the presheaf defined by 
$ U \mapsto \check H^q(j^{-1}(U);\Z_M)$.
The associated sheaf is the $q$-direct image sheaf $R^qj_*\Z_M$.

Since $\Z$ is a ring, the cup product $\check H^q(j^{-1}(U);\Z_M) \otimes \check H^{q'}(j^{-1}(U);\Z_M) \rightarrow \check H^{q+q'}(j^{-1}(U);\Z_M)$ is defined in \v Cech cohomology, for details see \cite[Section II.7]{Bredon}.
The cup product induces the map of sheaves $f_{q,q'}:R^qj_*\Z_M \otimes R^{q'}j_*\Z_M \rightarrow R^{q+q'}j_*\Z_M$.
In the same way we can define $R^qj_*\Q_M$.
We define the maps 
\begin{equation} \label{eq:def_smile}
\smile:\check H^p(T; R^qj_*\Z_M) \otimes \check H^{p'}(T; R^{q'}j_*\Z_M) \rightarrow \check H^{p+p'}(T; R^{q+q'}j_*\Z_M).
\end{equation}

as $(-1)^{p'q}$ times the composition of the cup product in the \v Cech cohomology and $f_{q,q'}$.

The inclusion $j$ defines a natural map in cohomology $H^\bullet(T) \to H^\bullet(M)$ which is injective, so we identify $H^\bullet(T)$ with its image.
We define a increasing filtration $F_\bullet=\{F_i\}_{i \in Z}$ for the cohomology ring $H^\bullet(M)$ by
\[ F_i = H^{\leq i}(M;\Z) \otimes H^\bullet(T;\Z)\]
for $i\geq 0$ and by $F_{-1}=0$.
The graded ring $\gr_{F_\bullet} H^\bullet(M;\Z)$ associated with the filtration $F_\bullet$  is the ring $\bigoplus_{i\geq 0} F_i/F_{i-1}$.

\begin{lemma}[\cite{Bredon}] \label{teo_LeraySS}
There exists a spectral sequence of $\Z$-algebras $E_n^{p,q}(M)$ which converges, as a bi-graded algebra, to $\gr_{F_\bullet}  H^\bullet(M;\Z)$.
The second page of the spectral sequence is
\[
E^{p,q}_2(M)=\check H^p(T; R^qj_*\Z_M)
\]
and the product coincides with the map defined in \eqref{eq:def_smile}.
\end{lemma}

\begin{proof}
The existence and the convergence of the spectral sequence are proven in \cite[IV, Theorem 6.1]{Bredon}.
The cup product in Leray spectral sequence is described in \cite[IV, section 6.8]{Bredon}.

The limit of the spectral sequence is a graded ring associated with a filtration of $H^\bullet(M;\Z)$ that can be determine as follows.
The Leray spectral sequence can be identified to the first (or horizontal) spectral sequence of an appropriate double complex.
The filtration in the double complex is described in \cite[A, section 2]{Bredon} and coincides with $F_\bullet$.
\end{proof}


From now on, let $j:M \rightarrow T$ be the open inclusion of complement of a toric arrangement in the corresponding torus, so the equality $j_* \Z_M= \Z_T$ holds.
The higher direct image sheaves $R^qj_* \Z_M$ and $R^qj_* \Q_M$ has been partially described in \cite{Looijenga} and in \cite{Bibby}, respectively.
The analogous of the following lemma for the sheaf $R^qj_*\Q_M$ has been proven in \cite[Lemma 3.1]{Bibby}.
We adapt the proof of \cite{Bibby} in order to study the cup product structure in the case of integer coefficients.

\begin{lemma}\label{lemma_Bibby}
Let $i_W$ be the inclusion $W \hookrightarrow T$ for $W\in \S$.
For all natural numbers $q$ there exists an isomorphism of sheaves:
\[\varphi_q: \bigoplus_{\rk W=q} (i_W)_* \Z_W \otimes_\Z H^q(M(\Delta[W]);\Z) \stackrel{\sim}{\longrightarrow} R^q j_* \Z_M \]
\end{lemma}

\begin{proof}
Recall that the sheaf $R^qj_* \Z_M$ is the sheafification of the presheaf $P$, defined by:
\[ P(U) = \check H^q(j^{-1}(U); \Z_M)\]
Let $x\in T$ be a point of the torus and denote with $W_x$ the smallest (with respect to inclusions) layer containing the point $x$.
There exists a neighborhood $V_x \subset T$ of $x$ which is diffeomorphic to the tangent space $\T _x T$ and the restriction from $V_x \cap M$ into $M(\Delta[x])$ is an isomorphism.
By construction we have $M(\Delta[x])=M(\Delta[W_x])$, so the stalk of  $R^qj_* \Z_W$ at $x$ is 
\[ (R^qj_* \Z_M)_x = H^q(M(\Delta[W_x]);\Z) \simeq \bigoplus_{\substack{\rk W=q \\ x \in W}} H^q(M(\Delta[W]);\Z)\]
where the last isomorphism is the Brieskorn isomorphism, see \cite[Theorem 3.26, p. 65]{orlik-terao} or \cite[Lemma 3, p.27]{Brieskorn}.
We define a sheaf on $T$, $\epsilon_W= (i_W)_* \Z_W \otimes_\Z H^{\rk W} (M(\Delta[W]);\Z)$,  for all $W\in \S$.
Let $n(U,W)$ be the set of connected components of $U\cap W$, for every open set $U\subseteq T$ we have:
\[ \epsilon_W(U)= \Z^{\oplus n(U,W)} \otimes_\Z H^{\rk W}(M(\Delta[W]);\Z) \]
For every $x \in T$ generic in $W$ and for every open subset $U \subseteq V_x$, we can define a map 
\[ H^{\rk W}(M(\Delta[W]);\Z) \rightarrow H^{\rk W}(U \cap M; \Z) \]
as the pullback of the natural inclusion $U \cap M \hookrightarrow M(\Delta[W_x])$.
These maps glue to a global map of sheaves $\varphi_W: \epsilon_W \rightarrow R^{\rk W}j_* \Z_M$.
Let $\varphi_q$ be the direct sum map from $\epsilon_q=\oplus_{\rk W=q} \epsilon_W$ into $R^qj_*\Z_M$.

We show that $\varphi_q$ is the desired isomorphism by checking the maps on the stalks.
In fact, for $x \in T$:
\[(\varphi_q)_x: (\epsilon_q)_x = \bigoplus_{\substack{\rk W=q \\ x \in W}}H^q(M(\Delta[W]);\Z) \rightarrow H^q(M(\Delta[W_x]);\Z)=(R^qj_* \Z_M)_x \]
is the Brieskorn isomorphism.
\end{proof}

In order to study the product map $f_{q,q'}$, we introduce the map 
\[b_{W,W',L}\colon \epsilon_W \otimes \epsilon_{W'} \to \epsilon_{L}\]
which is defined as follows:
if $L$ is not a connected component of $W\cap W'$
then is the zero map, otherwise we set:
\[b_{W,W',L}((\alpha \otimes a)\otimes (\gamma \otimes c))= (\alpha_{|L} \gamma_{|L}) \otimes (b_{W,L}(a) \cup b_{W',L}(c) )\]
where $b_{W,L}$ is the Brieskorn inclusion:
\[H^{\rk W}(M(\Delta[W]);\Z)\simeq H^{\rk W}(M(\Delta[L]_W);\Z) \hookrightarrow H^{\rk W}(M(\Delta[L]);\Z).\]
Now we can consider the direct sum map 
\[b_{q,q'}= \bigoplus_{\substack{\rk W=q \\ \rk W'=q' \\ \rk L= q+q'}} b_{W,W',L} \colon \epsilon_q \otimes \epsilon_{q'} \to \epsilon_{q+q'}.\]

\begin{lemma}\label{lemma_prod}
The isomorphism $\varphi$ of \cref{lemma_Bibby} is compatible with $f_{q,q'}$ and $b_{q,q'}$, i.e the diagram below commutes.
\begin{center}
\begin{tikzpicture}[scale=1]
\node (xa1) at (0,0) {$R^q j_* \Z_M \otimes_\Z R^{q'} j_* \Z_M $};
\node (xa2) at (4,0) {$R^{q+q'} j_* \Z_M $};
\node (x1) at (0,2) {$\epsilon_q \otimes \epsilon_{q'}$};
\node (x2) at (4,2) {$\epsilon_{q+q'}$};
\draw[->] (xa1) to node[above]{$f_{q,q'}$} (xa2);
\draw[->] (x1) to node[left]{$\varphi_q \otimes \varphi_q'$} (xa1);
\draw[->] (x2) to node[left]{$\varphi_{q+q'}$} (xa2);
\draw[->] (x1) to node[above]{$b_{q,q'}$} (x2);
\end{tikzpicture}
\end{center}
\end{lemma}

\begin{proof}
Let $U\subset T$ be an open set;
it is sufficient to show that $f_{q,q'} \circ \varphi_q \otimes \varphi_{q'}$ and $\varphi_{q+q'} \circ b_{q,q'}$ agree on the generators of $\epsilon_q \otimes \epsilon_{q'}(U)$.
Let $(\alpha \otimes a)\otimes (\gamma \otimes c)$ be an element of the ring $\epsilon_W \otimes \epsilon_{W'}(U)$ and $x$ a point of the torus $T$: we check the equality on the stalk at $x$.
If $x$ is not in $W \cap W'$ then $[(\alpha \otimes a)\otimes (\gamma \otimes c)]_x$ is zero and the assertion is obvious.
Otherwise, let $L$ be the connected component of $W \cap W'$ containing $x$.
From the fact that 
\[b_{L,W_x}\circ (b_{W,L} (a) \cup b_{W',L}(c))= b_{W,W_x} (a) \cup b_{W',W_x}(c)\]
we have that both stalks are $\alpha_{x}\gamma_x b_{W,W_x} (a) \cup b_{W',W_x}(c)$.
\end{proof}

In order to provide a description of the product structure in $\gr_{F_\bullet} H^\bullet(M;\Z)$ we introduce the following definition.

\begin{de}\label{prod}
Let $A^{\bullet,\bullet}(\Delta)$ be the bi-graded algebra whose homogeneous components are 
\[A^{p,q} \defeq \bigoplus_{\rk W =q} H^p(W;\Z) \otimes_\Z H^q(M(\Delta[W]);\Z).\]
The product map $\odot$ is defined on the generators $\alpha \otimes a \in H^p(W;\Z) \otimes_\Z H^{\rk W} (M(\Delta[W]);\Z)$ and $\gamma \otimes c  \in H^{p'}(W';\Z) \otimes_\Z H^{\rk W'}(M(\Delta[W']);\Z)$ by the formula:
\[[(\alpha \otimes a) \odot (\gamma \otimes c)]_L= (i_{W,L}^* \alpha \cup i_{W',L}^* \gamma) \otimes (b_{W,L}(a) \cup b_{W',L}(b)) \]
where $L$ is a connected components of $W\cap W'$ and $i_{W,L}$ (respectively $i_{W',L}$) the inclusion of $L$ in $W$ (respectively in $W'$).
\end{de}

\begin{teo}\label{teo_second_page}
The second page of the Leray spectral sequence of \Cref{teo_LeraySS} is isomorphic as a bi-graded algebra to $A^{\bullet,\bullet}$.
\end{teo}

\begin{proof}
The isomorphism $\varphi_q: \epsilon_q \rightarrow R^qj_* \Z_M$ of \cref{lemma_Bibby} induces an isomorphism in cohomology:
\[\tilde{\varphi}_q: \bigoplus_{\rk W =q} H^\bullet(W;\Z) \otimes_\Z H^q(M(\Delta[W]);\Z) \rightarrow E^{\bullet,q}_2(M)\]
Hence we have an isomorphism $\tilde{\varphi}: A^{\bullet, \bullet}(\Delta) \rightarrow E^{\bullet, \bullet}_2(M)$; \cref{lemma_prod} ensures then that $\tilde{\varphi}$ is an isomorphism of algebras.
\end{proof}

The next lemma appeared first in \cite{Callegaro}.
An analogue on the rationals was proven in \cite[Lemma 3.2]{Bibby} in a more general setting using some Hodge theory.

\begin{lemma}[{\cite[Theorem 5.1.3]{Callegaro}}]\label{lemma_SSdegenerate}
The Leray spectral sequence  for the inclusion $M \hookrightarrow T$ degenerates at the second page.
Hence the two algebras $E_2^{\bullet,\bullet}(M)$ and $\gr_{F_\bullet} H^{\bullet}(M;\Z)$ are isomorphic.
\end{lemma}

Up to changing the coefficients, the filtration $F_\bullet$ coincides with the one defined in \cite[Remark 4.3]{DeConcini}.
From now on, we denote by $\gr H^\bullet(M;\Z)$ the bi-graded, graded commutative, $\Z$-algebra associated to $ H^\bullet(M;\Z)$ with respect to the filtration $\{F^\bullet_n\}_{n \in \Z}$.

As a consequence of the previous statements we obtain the following result.
\begin{teo}\label{main_on_cohomology}
The Leray spectral sequence gives the following isomorphism of bi-graded $\Z$-algebras:
\[f: A^{\bullet,\bullet}(\Delta) \longrightarrow \gr H^\bullet(M(\Delta);\Z)\]
\end{teo}

\begin{proof}
The result follows since the map $f$ is the composition the isomorphism given in \cref{teo_second_page} between $A^{\bullet,\bullet}(\Delta)$ and $E^{\bullet,\bullet}_2(M)$ and the isomorphism of \cref{lemma_SSdegenerate}.
\end{proof}

\section{Uniqueness of representability}
\label{sect_On uniqueness of representability}
In this section we deal with representable arithmetic matroids and with the poset of intersections of a toric arrangement.
The aim is to show, under certain suitable hypotheses, that the representation of an arithmetic matroid is unique up to change of basis and sign reverse.
An analogue result is proven for posets, namely: if a poset is the poset of intersections of a centred toric arrangement then the associated toric arrangement is unique (up to torus automorphism).
The proofs of these assertions follow the ideas in \cite{Lenz} with the only difference that the entries of the matrices are rational numbers, instead of integers.

A \textit{matroid} is a finite \textit{ground set} $E$ together a \textit{rank function} $\rk \colon 2^E \to \N$, that satisfies some properties listed, for example, in \cite[Section 1.1.3]{OxleyBook}.
An \textit{arithmetic matroid} is a matroid together with a \textit{multiplicity function} $m \colon 2^E \to \N$ that satisfies five properties listed in \cite[Section 2.3]{DADDERIO} and in \cite[Section 2]{BM14}.

\begin{de}
An arithmetic matroid $(E,\rk, m)$ is said to be \textit{surjective} if $m(E)=1$ and \textit{torsion-free} if $m(\emptyset)=1$.
\end{de}

A list of elements $(v_e)_{e \in E}$ in a finitely generated abelian group $G$ define an arithmetic matroid with rank function $\rk(S)= \operatorname{rank} G/\langle v_s \rangle_{s \in S}$ and multiplicity function $m(S)$ equals to the cartinality of the torsion subgroup of $G/\langle v_s \rangle_{s \in S}$.

An arithmetic matroid is \textit{representable} if it came from a list of elements in such a way (see the exact definition in \cite{DADDERIO,BM14}).

\begin{de}
A toric arrangement $\Delta$ is said to be \textit{surjective} if the characters $\chi \in \Delta$ generate the lattice $\Lambda$. 
\end{de}

\begin{remark}
The property of toric arrangements to be surjective depends only on the poset of layers.
\end{remark}

Consider a vector space $V$, $\B$ one of its bases and $\{v_i\}_{i=1,\dots, n}$ a list of vectors: we denote by $M_{\B} (\{v_i\}_{i=1,\dots, n})$ the associated matrix of coordinates which is of size $\dim V \times n$.

Let $(E,\rk, m)$ be a representable, torsion-free, arithmetic matroid; choose a basis $\B \subseteq E$ and an essential representation given by the vectors $\{v_e\}_{e\in E}$ in $\Z^r$, where $r=\rk(E)$.
A basis of the vector space $\Q^r \supset \Z^r$ is given by $\{v_b\}_{b\in \B}$, hence we can write the vectors $\{v_e\}_{e \in E}$ as a matrix whose columns are the coordinates of each $v_e$ with respect to the basis $\{v_b\}_{b\in \B}$:
\[M_{\B} (\{v_e\}_{e\in E})= \left( \begin{array}{c|c}
\id_r & A
\end{array} \right) \in M(r,n;\Q)\]
where we have fixed an isomorphism between $E$ and $[n]$ that sends $\B\subset E$ into $[r]\subset [n]$.

Analogously, we can order the characters $\chi \in \Delta$ such that the first $r$ of them are a basis of the vector space $\Lambda_\Q \defeq \Lambda \otimes_\Z \Q$.
Set $E=[n]$ and $\B=[r]$ and write all characters in coordinates to obtain a matrix $\left( \begin{smallmatrix}
\id_r & | & A'
\end{smallmatrix} \right) \in M(r,n;\Q)$.
 
In the following discussion, $A\in M(r,n-r;\Q)$ denotes the matrix associated to representable arithmetic matroids or the one associated to toric arrangements.
Let $C\in M(r,n-r;\Z)$ be the matrix with $c_{i,j}=0$ if the corresponding element $a_{i,j}$ of $A$ is zero and $c_{i,j}=1$ otherwise.
The matrix $C$ is the incidence matrix of a bipartite graph $\G$
with $(r,n-r)$ vertices.
We denote the vertices by $R_1, \cdots, R_r$ and $C_1, \cdots, C_{n-r}$; two vertices $R_i$ and $C_j$ are connected with an edge $\alpha_{i,j}$ if and only if $c_{i,j}=1$.
Now choose a maximal forest $\A$ in $\G$, \ie a maximal subgraph without cycles.

\begin{lemma}\label{uniq_rat_matr}
Let $(E,\rk,m)$ be a representable, torsion-free, arithmetic matroid and consider a basis $\B$, a bijection $E\simeq [n]$, a maximal forest $\A$ in $\G$ and an essential representation $\{v_e\} \subset \Z^r$, as above.
Then:
\begin{enumerate}
\item Replacing some $v_e$ with $-v_e$ we obtain another representation of the arithmetic matroid which is given by a matrix $A'$ such that its element $a'_{i,j}$ is positive whenever $\alpha_{i,j} \in \A$.
\item If all entries $a_{i,j}$ of $A$ such that $\alpha_{i,j} \in \A$ are positive, the matrix $A$ is uniquely determined by the arithmetic matroid.
\end{enumerate}
\end{lemma}

\begin{lemma}\label{uniq_rat_poset}
Let $\Delta$ be an essential toric arrangement and consider $r$ characters that form a basis of $\Lambda_\Q$, an order of $\Delta$ 
and a maximal forest $\A$ in $\G$, as above.
Then:
\begin{enumerate}
\item Replacing some $\chi_i$ with $-\chi_i$ we obtain another toric arrangement $\Delta'$ such that $M(\Delta)=M(\Delta')$ and $\S(\Delta)=\S(\Delta')$, which associated matrix $A'$ has  element $a'_{i,j}$ positive whenever $\alpha_{i,j} \in \A$.
\item If all entries $a_{i,j}$ of $A$ such that $\alpha_{i,j} \in \A$ are positive, then the matrix $A$ is uniquely determined by the poset $\S$.
\end{enumerate}
\end{lemma}

\begin{de}
A toric arrangement $\Delta$ is in \textit{normal form} with respect to a basis $\B\subset E$ and to a maximal forest $\A$ if all entries $a_{i,j}$ of the matrix $A$ such that $\alpha_{i,j} \in \A$ are positive. 
\end{de}

Let $\Delta=\{(a_e,\chi_e)\}_{e \in E}$ be a toric arrangement: define a function $m$ from the subsets of $E$ to the natural numbers as follows:
\[ m(I)= \# \left\{\textnormal{connected components of } \bigcap_{i \in I} H_i \right\} \quad \forall I \subseteq E \]
This function is determined by the poset $\S$, in fact it coincides with the number of elements $W$ in $\S_{\rk I}$ such that $W> H_i$ for all $i \in I$.

\begin{example}\label{ex:running}
Let $M=(E=\{1,2,3\},\rk)$ be the matroid of three distinct lines in the real plane.
The function $m$ defined by:
\begin{align*}
& m(\emptyset)=1\\
& m(e)=1  \quad \textnormal{for } e=1,2,3 \\
& m(1,2)= 10 \\
& m(1,3)= 15 \\
& m(2,3)= 25 \\
& m(1,2,3)=5
\end{align*}
makes the matroid $M$ an arithmetic matroid $([3], \rk, m)$.
This arithmetic matroid is representable, indeed a possible representation is given by the matrix $\left( \begin{smallmatrix}
-2 & -32 & -43 \\
1 & 21 & 29
\end{smallmatrix} \right)$.

Choose a basis of the matroid, \eg $\B =\{1,2\} $ consider the matrix $A\in M(2,1; \Q)$ representing the coordinates in the basis $\B$ of the third element of $E$. 
The absolute value of the entries of $A$ is easy to determine:
\begin{align*}
|a_{1,1}| &= \frac{m(2,3)}{m(1,2)} = \frac{5}{2} \\
|a_{2,1}| &= \frac{m(1,3)}{m(1,2)} = \frac{3}{2}
\end{align*}
The associated matrix $C$ 
is then $\left( \begin{smallmatrix}
1 \\ 1
\end{smallmatrix} \right)$, and the associated bipartite graph is:

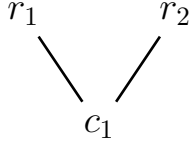
\begin{center}
\begin{tikzpicture}[shorten >=1pt, auto, ultra thick]
    \tikzstyle{node_style} = [font=\sffamily\Large\bfseries]
    \tikzstyle{edge_style} = [draw=black, line width=1]
    \node[node_style] (v1) at (-1,1) {$r_1$};
    \node[node_style] (v2) at (1,1) {$r_2$};
    \node[node_style] (v3) at (0,-0.5) {$c_1$};
    \draw[edge_style]  (v1) edge (v3);
    \draw[edge_style]  (v2) edge (v3);   
\end{tikzpicture}
\end{center}
This graph has a unique maximal tree, that we call $\A$, hence the normal form of the matrix $A$ (in normal form) is $\left( \begin{smallmatrix}
\frac{5}{2} \\ \frac{3}{2}
\end{smallmatrix} \right)$.
A representation of the arithmetic matroid in normal form is given by $\left( \begin{smallmatrix}
2 & -32 & -43 \\
-1 & 21 & 29
\end{smallmatrix} \right)$
which is obtained from the one we had before by changing the sign of the first column.
\end{example}

\begin{lemma} \label{lemma_abs_det}
The absolute values of the determinants of square submatrices of $A$ are uniquely determined by the underlying arithmetic matroid, if $A$ is defined by an arithmetic matroid, or by the poset of layers, if $A$ is defined by a toric arrangement.
\end{lemma}

\begin{proof}
Let $A_{I,J}$ be the square submatrix of $A$ with $I\subseteq \B$ and $J \subseteq E\setminus \B$.
We claim that:
\[ |\det (A_{I,J})|= \begin{cases}
\frac{m(\B\setminus I \cup J)}{m(\B)} & \textnormal{if } \B\setminus I \cup J \textnormal{ is a basis} \\
0 & \textnormal{otherwise}
\end{cases} \]
where $m$ is the multiplicity function in the case of arithmetic matroids and the function defined above in the case of toric arrangements.

Let $H$ be the square submatrix of $\left( \begin{smallmatrix}
\id_r & | & A
\end{smallmatrix} \right)$ made by the columns indexed by $\B\setminus I \cup J$.
Clearly $|\det (A_{I,J})|= |\det H|$.
If $\B\setminus I \cup J$ is a dependent set, then $\det H=0$, otherwise we compute $|\det H|$ using the formula
\[ M_{\mathcal{C}}(\{v_b\}_{b \in \B}) M_{\B}(\{v_e\}_{e \in \B\setminus I \cup J})=M_{\mathcal{C}}(\{v_e\}_{e \in \B\setminus I \cup J})\]
where $\mathcal{C}$ is the canonical basis of $\Z^r$ and $M_{\mathcal{C}}(\{v_b\}_{b \in \B})$ is the matrix of coordinates of the vectors $\{v_b\}_{b \in \B}$ in the basis $\mathcal{C}$.
Notice that $|\det M_{\mathcal{C}}(\{v_b\}_{b \in B})|= m(B)$ for all basis $B$ of $E$, thus 
\[m(\B)|\det H|= m(\B \setminus I \cup J)\]
Since $m$ depends only on the poset $\S$, we obtain the statement of the lemma.
\end{proof}

\begin{cor}\label{lenz_lemma7}
The rational numbers $|a_{i,j}|$ are uniquely determined by the arithmetic matroid (or respectively by the poset).
\end{cor}

\begin{proof}
The number $|a_{i,j}|$ is the absolute value of the determinant of the submatrix $A_{i,j}$ of size $1 \times 1$.
The result follows from \cref{lemma_abs_det}.
\end{proof}

\begin{proof}[Proof of \Cref{uniq_rat_matr} and \Cref{uniq_rat_poset}]
In both cases we deal with a matrix $A\in \M(r,n-r;\Q)$ and with the action of $(\Z/2\Z)^r \times (\Z/2\Z)^{n-r}$ on $\M(r,n-r;\Q)$:
this group acts with sign reverse of rows and columns.
The first assertion was proven in \cite[Lemma 6]{Lenz} in the case that $A$ has integer coefficients, but the proof never uses the fact that the coefficients are integers and works also in the case of rational coefficients.

The uniqueness was substantially proven in \cite[Lemma 9]{Lenz}, but since we are dealing with rational coefficients, we need to use \cref{lenz_lemma7} and \cref{lemma_abs_det} instead of \cite[Lemma 7]{Lenz} and \cite[Lemma 8]{Lenz}, respectively.
\end{proof}

\begin{teo}\label{teo_uniq_repr}
Let $(E,\rk, m)$ be a representable, surjective, torsion-free arithmetic matroid.
Then there exists an unique essential representation of $(E,\rk, m)$ up to sign reverse of columns and up to automorphism of $\Z^r$.
\end{teo}

\begin{proof}
Let $\{v_e\}_{e\in E}$ in $\Z^r$ be an essential representation of $(E,\rk,m)$, $X=M_{\mathcal{C}}(\{v_e\}_{e\in E})$ the matrix of coordinates with respect to the canonical basis.
The statement is equivalent to showing that for any other representation $X'=M_\C(\{v'_e\}_{e \in E})$ there exist two matrices $G\in \gl(r;\Z)$ and $D\in \gl(n;\Z)$, such that $D$ is diagonal with all entries $d_{i,i}=\pm 1$ and $X'=GXD$.

\Cref{uniq_rat_matr} proves the existence of the matrices $G \in \gl(r;\Q)$ and $D\in \gl(n;\Z)$ with $d_{i,i}=\pm 1$; we want to show that $G$ has integer coefficients.
The surjectivity assumption is equivalent to $\langle v_e \rangle_{e \in E} =\Z^r=\langle v'_e \rangle_{e \in E}$, while matrix $G$ sends the lattice $\langle v_e \rangle_{e \in E} \subset \Q^r$ isomorphically into $\langle v'_e \rangle_{e \in E} \subset \Q^r$.
Since $G\Z^r= \Z^r$, $G$ has integer coefficients.
\end{proof}

\begin{cor} \label{cor_unic_centred_TA}
Let $\Delta$ be a centred, surjective, essential toric arrangement in $T$ with poset of layers $\S$.
For any other essential toric arrangement $\Delta'$ in $T$ with the same poset of layers $\S$, there exists an automorphism $g$ of the torus $T$ that sends the open set $M(\Delta')$ homeomorphically into $M(\Delta)$.
\end{cor}

The proof of \cref{cor_unic_centred_TA} is analogous to that of \cref{teo_uniq_repr} except for the fact that $X=M_\mathcal{C}(\{ \chi\}_{\chi \in \Delta})$ and that we must apply \cref{uniq_rat_poset} instead of \cref{uniq_rat_matr}.

Given a set of vectors $\{v_i\}_{i=1, \dots, n}$ in $\Z^r$, some linear relation among them may hold; a relation can be written in the form
\begin{equation}\label{eq_relazioni}
\sum_{i \in I}c_i m_i v_i =0
\end{equation}
for some $I \subseteq E$, $c_i\in \{\pm 1\}$ and some $m_i$ positive, relatively prime integers (\ie $\gcd_{i\in I} (m_i)=1$).

\begin{lemma}
Let $(E,\rk,m)$ be a representable, torsion-free, arithmetic matroid, consider, as before, a basis $\B$, a bijection $E\simeq [n]$ and a maximal forest $\A$ in $\G$.
Let $\{v_e\}_{e \in E} \subset \Z^r$ be an essential representation such that the element $a_{i,j}$ of $A$ are positive whenever $\alpha_{i,j} \in \A$.
Then the linear relations among these vectors are uniquely determined by the arithmetic matroid. 
\end{lemma}

\begin{proof}
Equation \eqref{eq_relazioni} holds if and only if
\[\left( \begin{array}{c|c}
\id_r & A
\end{array} \right)\left( \begin{array}{c}
c_1 m_1 \\ \vdots \\ c_n m_n
\end{array} \right)=0\]
or, equivalently, the vector $w=\left( \begin{smallmatrix}
c_1 m_1 \\ \vdots \\ c_n m_n
\end{smallmatrix} \right)$ is in the kernel of the linear application $\left( \begin{smallmatrix}
\id_r & | & A
\end{smallmatrix} \right)$.
Since the matrix $A$ is uniquely determined by the arithmetic matroid (\cref{uniq_rat_matr}), this implies that so are all the relations.
\end{proof}

\begin{teo}\label{teo_rel_chars}
Let $\Delta$ be an essential toric arrangement in normal form with respect to a basis $\B$ and to a maximal forest $\A$.
The linear relations $\sum_{i\in I} c_im_i \chi_i=0$ among the characters, $\chi_i \in \Delta$, are uniquely determined by the poset $\S$. 
\end{teo}

\begin{proof}
The relation $\sum_{i\in I} c_im_i \chi_i=0$ holds if and only if the vector $\left( \begin{smallmatrix}
c_1 m_1 \\ \vdots \\ c_n m_n
\end{smallmatrix} \right)$ is in the kernel of  the linear application $\left( \begin{smallmatrix}
\id_r & | & A
\end{smallmatrix} \right)$.
Since the matrix $A$ is uniquely determined by the poset (\cref{uniq_rat_poset}), this implies that so are all the relations.
\end{proof}

\section{Graded toric Orlik-Solomon algebra}
\label{sect_OS_algebra}
In this section we give a purely combinatorial presentation of the bigraded algebra $\gr H^\bullet(M(\Delta);\Q)= A^{\bullet, \bullet}(M(\Delta)) \otimes_\Z \Q$.
We begin by defining the model $GOS^{\bullet,\bullet}(M(\Delta);\Q)$, then we exhibit a basis of this $\Q$-vector space and finally showing an isomorphism between the objects $\gr H^\bullet(M(\Delta);\Q)$ and $GOS^{\bullet,\bullet}(M(\Delta);\Q)$.

Now on we need to fix a total order on $E$.
A subset $S$ of $E$ is \textit{dependent} (respectively \textit{independent}) if the characters $\{\chi_i\}_{i\in S}$ are linearly dependent (respectively linearly independent).
A minimal dependent set $C$ is said a \textit{circuit} and the set of the form $C \setminus \{c\}$, where $c$ is the maximal element of the circuit $C$, is a \textit{broken circuit}.
A \textit{no broken circuit} is a independent set not containing any broken circuit.

We recall the definitions of \cite{DeConcini}.
\begin{de}
Let $E$ be the ordered set of indexes for the elements of $\Delta$ and let $W$ be a layer.
A $k$-tuple $S=(s_1, \cdots, s_k)$ with $s_i \in E$ is \textit{standard} if the sequence is increasing.
The set $S$ is a \textit{circuit associated to} $W$ if it is a minimal dependent set and $W$ is a connected component of $\bigcap_{i \in S} H_i$.
The set $S$ is an \textit{independent set associated to} $W$ (we will simply say that $S$ is \textit{associated to} $W$) if it is independent and $W$ is a connected component of $\bigcap_{i \in S} H_i$.
The set $S$ is a \textit{no broken circuit associated to} $W$ if it is a standard no broken circuit and $W$ is a connected component of $\bigcap_{i \in S} H_i$.
\end{de}

We recall a well-known theorem of Orlik-Solomon (in \cite{Orlik-Solomon}) on hyperplanes arrangements.
Let $\mathcal{H}$ be a central hyperplane arrangement with hyperplanes indexed by the ordered set $E$.
Let $W^\bullet$ be the exterior graded algebra on generators $\{f_i\}_{i \in E}$ over the ring of integers.
The Orlik-Solomon algebra $OS^\bullet(\mathcal{H};\Z)$ is the quotient of $W^\bullet$ by the ideal generated by the equations:
\[ \sum_{i \in S} (-1)^i f_{S\setminus s_i}=0\]
for all circuits $S$ of $E$, where we adopt the usual notation $f_T= f_{t_1}f_{t_2} \dots f_{t_k}$.

\begin{teo}[Orlik-Solomon \cite{Orlik-Solomon}]
For every central hyperplane arrangement $\mathcal{H}$ there is a an isomorphism:
\[OS^\bullet(\mathcal{H};\Z) \xrightarrow{\sim} H^\bullet(M(\mathcal{H});\Z)\]
Moreover a basis of $OS^\bullet(\mathcal{H};\Z)$ as $\Z$-module is given by the elements $f_S$ where $S$ is a no broken circuit.
\end{teo}

For every toric arrangement $\Delta$ in normal form, let us define a bigraded $\Q$-vector space $V$ with basis given by the formal symbols  
\[\{e_T y_{W,S} \mid S,T \subseteq E,\, W \in \S,\, S \textnormal{ is associated to }W\}.\]
Define the grade of $e_Ty_{W,S}$ to be $(|T|, \rk W)$, or, equivalently, $(|T|, |S|)$.
We will write $e_T$ instead of $e_T y_{U, \emptyset}$ (where $U$ is the torus, \ie the unique layer of rank zero) and $y_{W,S}$ instead of $e_\emptyset y_{W,S}$.
We endow $V$ with the exterior product defined on the generators by
\[ e_Ty_{W,S} \wedge e_{T'}y_{W',S'}= 0 \]
if the tuple $SS'$ is dependent and by
\[ e_Ty_{W,S} \wedge e_{T'}y_{W',S'} =  (-1)^{|S||T'|}\sum_{W''} e_{TT'}y_{W'',SS'} \]
otherwise, where $W''$ runs over all connected components of $W \cap W'$ and $SS'$ is the concatenation of $S$ and $S'$.
The vector space $V$ endowed with product is a bigraded algebra over $\Q$.

Let $J$ be the homogeneous ideal generated by the following relations:
\begin{enumerate}
\item $e_{\tau T}y_{W,\sigma S}=(-1)^{\sgn \sigma+ \sgn \tau}e_T y_{W,S}$ for all permutations $\sigma \in \mathfrak{S}_{|S|}$ and $\tau \in \mathfrak{S}_{|T|}$. \label{eq_reorder}
\item $e_iy_{W,S}=0$ for all $i\in S$. \label{eq_ovvia}
\item $\sum_{i=0}^{\rk W} (-1)^i y_{W,S\setminus s_i}=0$ for all circuits $S=\{s_0, \cdots, s_{\rk W}\}$ associated to $W$. \label{eq_alla_OS}
\item $\sum_{i\in I} c_i m_i e_i=0$ for all $I$, $c_i$ and $m_i, \, i \in I$ appearing in the relations of \cref{teo_rel_chars}. \label{eq_nel_toro}
\end{enumerate}

\begin{de}\label{def_GOS}
Let $\Delta$ be a toric arrangement in normal form.
The \textit{graded Orlik-Solomon algebra} for toric arrangement $\Delta$ is
\[ GOS^{\bullet,\bullet} (M(\Delta) ;\Q) = \faktor{V}{J}\]
\end{de}

Define $I_W$ as the set of characters in $\Lambda$ that vanish on the translated to the origin of the layer $W\in \S$.
For every layer $W\in \S$, we choose arbitrary $\dim W$ characters in $\Delta$ such that their projection in $\Lambda / I_W \otimes_\Z \Q$ form a basis as $\Q$ vector space.
Let $\mathcal{B}_W \subset E$ be the set of the chosen characters for the layer $W$.

\begin{lemma}\label{lemma_generators}
Let $\Delta$ be an essential toric arrangement.
Then a set of generators of $GOS^{\bullet,\bullet}(M(\Delta);\Q)$ as $\Q$-vector space is given by the elements $e_T y_{W,S}$ with $S$ no broken circuit associated to $W$ and $T$ a subset of $\mathcal{B}_W$.
\end{lemma}

\begin{proof}
It is sufficient to show that any generic element $e_Ty_{W,S}$ can be written as linear combination of the ones with $S'$ no broken circuit associated to $W'$ and $T'$ a subset of $\mathcal{B}_{W'}$.
By relation \ref{eq_reorder}, we may assume that $S$ and $T$ are standard lists and by relations \ref{eq_alla_OS} we can write each $e_Ty_{W,S}$ as sum of certain $e_Ty_{W,S'}$ with $S'$ a no broken circuit associated to $W$.
By relations \ref{eq_nel_toro}, we can write each $e_Ty_{W,S'}$ as linear combination of some $e_{T'}y_{W,S'}$ with $T'\subseteq \mathcal{B}_W$ or $T' \cap S' \neq \emptyset$.
In the case $T' \cap S' \neq \emptyset$, the relations of point \ref{eq_ovvia} show that $e_{T'}y_{W,S'}=0$ in $GOS^{\bullet,\bullet}(M(\Delta);\Q)$.
\end{proof}

Let $\omega$ be a generator of $H^1(\C^*;\Z)$.
The map $\chi : T \rightarrow \C^*$ induces a map in cohomology $\chi^*: H^1(\C^*;\Z) \rightarrow H^1(T;\Z)$ and gives  $\chi^*(\omega) \in H^1(T;\Z)$.
Let $\tilde{f}$ be the linear map $\tilde{f}: V \rightarrow A^{\bullet,\bullet} \otimes_\Z \Q$ defined on the generators $e_Ty_{W,S}$ by
\[ \tilde{f}(e_Ty_{W,S})=\chi_{t_1}^*(\omega)\chi_{t_2}^*(\omega)\cdots \chi_{t_p}^*(\omega) \otimes f_S \in H^p(W;\Z)\otimes H^{\rk W}(M(\Delta[W]);\Z) \]
where the canonical projection from $H^p(T;\Z) \rightarrow H^p(W;\Z)$ is understood and the element $f_S$ is a generator of the Orlik-Solomon algebra of the hyperplane arrangements $\Delta[W]$.


\begin{lemma}
The map $\tilde{f}$ is a morphism of algebras and factors through the algebra $GOS^{\bullet,\bullet}(M(\Delta);\Q)$.
\begin{center}
\begin{tikzpicture}[scale=1.6]
\node (x) at (0,1) {$V$};
\node (z) at (2,1) {$A^{\bullet,\bullet} \otimes_\Z \Q$};
\node (y) at (1,0) {$GOS^{\bullet,\bullet}(M(\Delta);\Q)$};
\draw[->] (x) to node[above]{$\tilde{f}$} (z);
\draw[->] (x) to (y);
\draw[->] (y) to (z);
\end{tikzpicture}
\end{center}
\end{lemma}

\begin{proof}
The map $\tilde{f}$ is a morphism of algebras, since 
\[ \tilde{f}(e_Ty_{W,S} \cdot e_{T'}y_{W',S'})= \tilde{f}(e_Ty_{W,S}) \odot \tilde{f}(e_{T'}y_{W',S'}), \]
where $\odot$ is the product of $A^{\bullet,\bullet}$ introduced in \Cref{prod}.
Moreover, the map $\tilde{f}$ sends the relations \ref{eq_reorder}-\ref{eq_nel_toro} to zero, hence the ideal $J$ is mapped into zero.
\end{proof}

Call $f: GOS^{\bullet,\bullet}(M(\Delta);\Q) \rightarrow A^{\bullet,\bullet}(M(\Delta)) \otimes_\Z \Q$ the map induced by $\tilde{f}$ on the quotient.

\begin{teo}\label{thm:main_on_rational}
Let $\Delta$ be an essential toric arrangement. 
The map $f$ is an isomorphism of algebras.
\[f: GOS^{\bullet,\bullet}(M(\Delta);\Q) \xrightarrow{\sim} A^{\bullet,\bullet}(M(\Delta)) \otimes_\Z \Q\]
\end{teo}

\begin{proof}
The map $f$ is surjective because $H^\bullet(W;\Z) \otimes_\Z \Q$ is generated in degree one by the elements $\chi^*(\omega)$ with $\chi \in \Delta$.

The set of generators described in \cref{lemma_generators} is of cardinality equal to the rank of $A^{\bullet,\bullet}$, so the map $f$ must be an isomorphism.
Consequently, the elements $e_Sy_{W,T}$ with $T$ no broken circuit associated to $W$ and $S\subseteq \mathcal{B}_W$ form a basis of $GOS^{\bullet,\bullet}$ as $\Q$-vector space.

Let $\mathcal{N}{q}$ be the following set:
\[\mathcal{N}{q} \defeq \{(W,S)\in \S_q \times \mathcal{P}_q(E) \mid S \textnormal{ is a no broken circuit of }\S_{\leq W} \}. \]
As shown for the first time by Looijenga in \cite{Looijenga} and then by De Concini and Procesi in \cite{DeConcini}, the Poincaré polynomial of $M(\Delta)$ is:
\begin{equation}\label{poly_char}
P_\Delta (t)= \sum_{q=0}^r |\mathcal{N}_q| (1+t)^{r-q}t^q
\end{equation}

The number of generators is equal to:
\[ \sum_{q=0}^{r} 2^{r-q}|\mathcal{N}_q|\]
which is exactly the value $P_\Delta(1)$ of the Poincaré polynomial.
\end{proof}

\begin{example}
Consider the toric arrangement in $(\C^*)^2$ of four hypertori having equations $H_1=\{x=1\}$, $H_2=\{y=1\}$, $H_3=\{xy^3=1\}$ and $H_4= \{xy^2= \zeta_3\}$, where $\zeta_3$ is a $3^{\textnormal{th}}$ root of unity.
The layers of rank two are $5$ points:
\begin{align*}
p &=(1,1) = H_1 \cap H_2 \cap H_3 \\
q &=(1, \zeta_3^2)=H_1 \cap H_3 \cap H_4 \\
r &=(1, \zeta_3) \subset H_1 \cap H_3 \\
s &=(\zeta_3,1) = H_2 \cap H_4 \\
t &=(1, -\zeta_3) \subset  H_1 \cap H_4
\end{align*}

An additive basis for the cohomology algebra of the complement is formed by
\[ e_1, \, e_2,\, y_1,\, y_2,\, y_3,\, y_4\]
in degree one and by
\[e_{12},\, e_2y_1,\, e_1y_2,\, e_2y_3,\, e_2y_4,\, y_{p,13},\, y_{p,23},\, y_{q,14},\, y_{q,34},\, y_{r,13},\, y_{s,24},\, y_{t,14} \]
in degree two.
\end{example}

\section{Coverings of tori}
\label{sect_coverings of tori}

Let $F$ be a finitely generated free $\Z$-module and $G$ a finite abelian group.
The extensions $X$:
\[ 0 \rightarrow F \rightarrow X \rightarrow G \rightarrow 0\]
are parametrized by $\Ext^1(G,F)$ up to equivalence.
We want to characterize the cardinality of the torsion subgroup $|\Tor X|$.
The torsion subgroup is isomorphic to its image in $G$ that we will call $H$.

\begin{lemma}
Let $F$ be a free $\Z$-module. Then the contravariant functor $\Ext^1( \cdot, F)$ from finite abelian groups to $\Z$-modules is an exact functor.
\end{lemma}

\begin{proof}
A short exact sequence of finite abelian groups $0 \rightarrow H \xrightarrow{i} G \xrightarrow{p} K \rightarrow 0 $ produces a short exact sequence
\[0 \rightarrow \Ext^1 (K,F) \xrightarrow{p^*} \Ext^1 (G,F) \xrightarrow{i^*} \Ext^1 (H,F) \rightarrow 0 \]
since $\Hom (H,F)=0$ and $\Ext^2(K,F)=0$ for all free $\Z$-module $F$ and all finite group $H$.
\end{proof}

\begin{lemma}\label{lemma_max_subgr}
Let $F$ be a free $\Z$-module, $G$ a finite abelian group and $x$ an element of $\Ext^1(G,F)$. Then:
\begin{enumerate}
\item There exists a unique subgroup $H\hookrightarrow G$,
 maximal among all subgroups $H'$ such that $i_{H'}^*(x)=0$.
\item There exists a unique quotient $G \twoheadrightarrow K$, minimal among all quotients $K'$ such that $x \in \im p_{K'}^*$.
\end{enumerate}
Moreover, such groups form an exact sequence $0 \rightarrow H \rightarrow G \rightarrow K \rightarrow 0 $.
\end{lemma}

\begin{proof}
Suppose that, for two subgroups $H$ and $H'$ of $G$, $i_H^*(x)=0$ and $i_{H'}^*(x)=0$.
There is a surjection $H \times H' \twoheadrightarrow HH' <G$ that gives an inclusion $\Ext^1(HH',F) \hookrightarrow \Ext^1(H,F) \times \Ext^1(H',F)$.
The element $i_{HH'}^*(x)$ maps to $(0,0)$ so it must be zero ($i^*_{HH'}(x)=0$ in $\Ext^1(HH',F)$).
The arbitrariness of $H$ and $H'$ gives the first result.

The second point follows from the first making use of the following fact:
for every exact sequence $0 \rightarrow H' \rightarrow G \rightarrow K' \rightarrow 0 $ the element $i^*_{H'}(x)$ is zero if and only if $x\in \im p^*_{K'}$.
\end{proof}

Let $T \rightarrow U$ be a covering of connected tori whose Galois group is the dual of $G$. 
We want to apply the above observations to this group $G$.
Call $\Lambda$ and $\Gamma$ the characters groups of $T$ and $U$ respectively, so that $G$ is isomorphic to $\Lambda / \Gamma$.
Let $\Delta_U=\{(a_e, \chi_e)\}_{e \in E}$ (with $\chi_e \in \Gamma$, $e \in E$) be a surjective toric arrangement in $U$ and $\Delta_T$ the lifted toric arrangement (\ie the arrangement $\{(a_e,\chi_e) \}_{e \in E}$ where  $\chi_e \in \Lambda$ now are thought as characters of $T$). 
We have the two associated multiplicity functions $m_U$ and $m_T$ from $\mathcal{P}(E)$ to $\N$.

\begin{remark}\label{remark_on_G}
For essential toric arrangements the function $m_T$ determines the isomorphism class of the group $G$, hence it depends only on the poset of layers $\S$ or on the arithmetic matroid (in the centred case).
In fact $G$ is the cokernel of the matrix $A$ whose columns are the coordinate of $\chi_e$ in some basis.
Then by the Smith normal form, its isomorphism class depends only on the greatest common divisor of the determinants of the minors of $A$.

The group $G$ can be presented as the cokernel of $D$, where $D$ is the $r\times r$ diagonal matrix with entries $d_i=\frac{e_i}{e_{i-1}}$, where $e_i= \gcd \{ m_T(E) \mid |E|=i \}$.
\end{remark}

We assume only the data of $\Delta_U$ (hence the datum $m_U$) and $m_T$ (the isomorphism class of $G$) are known.
Our purpose is to give an answer to the following question:
how many extensions $\Gamma \hookrightarrow \Lambda'$ giving the combinatorial data of $\Delta_T$ exist?
In order to find the answers, we make the following construction:
for $S\subseteq E$, define $\Gamma_S$ to be the subgroup of $\Gamma$ generated by the characters $\{\chi_e\}_{e\in S}$ and $\Rad_\Gamma (\Gamma_S)$ its radical in the lattice $\Gamma$.
Let $F_S$ be the quotient of $\Gamma$ by $\Rad_\Gamma (\Gamma_S)$ and notice that $F_S$ is a free $\Z$-module.
The exact sequence $0 \rightarrow \Rad_\Gamma (\Gamma_S) \rightarrow \Gamma \rightarrow F_S \rightarrow 0$ of free modules gives, for any finite abelian group $G$ the exact sequence:
\[ 0 \rightarrow \Ext^1(G,\Rad_\Gamma (\Gamma_S)) \rightarrow \Ext^1(G, \Gamma) \xrightarrow{\pi_S} \Ext^1(G,F_S) \rightarrow 0\]
\begin{de}
Consider $x\in \Ext^1(G,\Gamma)$ and call $H_S(x)$ the maximal subgroup $H$ of $G$ given by \cref{lemma_max_subgr} for the elements $\pi_S(x) \in \Ext^1(G,F_S)$.
An element $x\in \Ext^1(G,\Gamma)$ is said \textit{coherent with $\Delta_U$ and $m_T$} if for all $S \subseteq E$ we have $|H_S(x)|= \frac{m_T(S)}{m_U(S)}$.
\end{de}

Call $\mathcal{C}$ the subset of $\Ext^1(G,F_S)$ made by all elements coherent with $\Delta_U$ and $m_T$.
Recall that there is a natural right action of $\Aut(G)$ on $\Ext^1(G,\Gamma)$ and the set $\mathcal{C}$ is an invariant subset for this action.

\begin{lemma} \label{lemma_rivest_tori}
Let $U$ be a torus and $\Gamma$ its character group.
For all finite groups $G$ there exists a correspondence:
\begin{align*}
\faktor{\Ext^1(G,\Gamma)}{\Aut(G)} & \xleftrightarrow{1:1} \{T \rightarrow U \mid \textnormal{with Galois group dual to }G\}\\
[x] & \longmapsto (\Hom (x, \C^*) \to U)
\end{align*}
where $T$ is a torus ($T$ is not necessarily connected).
\end{lemma}

\begin{teo} \label{teo_corr_ext}
Let $\Delta_U$ and $m_T$ be as above.
There exists a natural correspondence
\[ \faktor{\mathcal{C}}{\Aut(G)} \xleftrightarrow{1:1} \{T' \rightarrow U \mid m_{T'}=m_T \} \]
where $m_{T'}$ is the multiplicity function of the arrangement $\Delta_U$ lifted to $T'$ and the right hand side is considered up to isomorphism over $U$.
The bijection is the restriction of the one in \Cref{lemma_rivest_tori}.
\end{teo}

\begin{proof}
A standard fact from commutative algebra is the correspondence between the group $\Ext^1(G,F)$ and the extension of the two $\Z$-modules $F$ and $G$:
\[ 0 \rightarrow F \rightarrow X \rightarrow G \rightarrow 0\]
up to equivalence, \ie two extensions $X$ and $X'$ are equivalent if there exists a commutative diagram 
\begin{center}
\begin{tikzpicture}[scale=1.6]
\node (a1) at (-1,0) {$0$};
\node (x) at (0.75,0.5) {$\circlearrowleft$};
\node (x) at (2.25,0.5) {$\circlearrowleft$};
\node (a0) at (-1,1) {$0$};
\node (b0) at (0,1) {$F$};
\node (b1) at (0,0) {$F$};
\node (c1) at (1.5,0) {$X'$};
\node (c0) at (1.5,1) {$X$};
\node (d1) at (3,0) {$G$};
\node (d0) at (3,1) {$G$};
\node (e1) at (4,0) {$0$};
\node (e0) at (4,1) {$0$};
\draw[->] (a1) edge (b1) (b1) edge (c1)(c1) edge (d1)(d1) edge (e1);
\draw[->] (a0) edge (b0) (b0) edge (c0)(c0) edge (d0)(d0) edge (e0);
\draw[double equal sign distance] (b0) -- (b1);
\draw[double equal sign distance] (d0) -- (d1);
\draw[->] (c0) -- (c1);
\end{tikzpicture}
\end{center}
Let $x \in \mathcal{C}$ be a element coherent with $\Delta_U$ and $m_T$; this element gives an extension $0 \rightarrow \Gamma \rightarrow \Lambda_x \rightarrow G \rightarrow 0$.
If $S\subseteq E$ then there is a commutative diagram:
\begin{center}
\begin{tikzpicture}[scale=1.6]
\node (x) at (0.75,0.5) {$\circlearrowleft$};
\node (x) at (2.25,0.5) {$\circlearrowleft$};
\node (a1) at (-1,0) {$0$};
\node (a0) at (-1,1) {$0$};
\node (b0) at (0,1) {$\Gamma$};
\node (b1) at (0,0) {$F_S$};
\node (c1) at (1.5,0) {$\faktor{\Lambda_x}{\Rad_{\Gamma} \Gamma_S}$};
\node (c0) at (1.5,1) {$\Lambda_x$};
\node (d1) at (3,0) {$G$};
\node (d0) at (3,1) {$G$};
\node (e1) at (4,0) {$0$};
\node (e0) at (4,1) {$0$};
\draw[->] (a1) edge (b1) (b1) edge (c1)(c1) edge (d1)(d1) edge (e1);
\draw[->] (a0) edge (b0) (b0) edge (c0)(c0) edge (d0)(d0) edge (e0);
\draw[->] (b0) -- (b1);
\draw[double equal sign distance] (d0) -- (d1);
\draw[->] (c0) -- (c1);
\end{tikzpicture}
\end{center}
Call $\Lambda_S$ the quotient $\Lambda_x / \Rad_{\Gamma} \Gamma_S$; we will show that the group $H_S (x)$ is the torsion subgroup of $ \Lambda_S$.
The exact sequence $0 \rightarrow F_S \rightarrow \Lambda_S \rightarrow G \rightarrow 0$ is represented by the element $\pi_S(x)$, therefore for all $G'$ subgroup of $G$, $i_{G'}(\pi_S(x))$ is zero if and only if the upper short exact sequence of the next diagram splits.
\begin{center}
\begin{tikzpicture}[scale=1.6]
\node (x) at (0.75,0.5) {$\circlearrowleft$};
\node (x) at (2.25,0.5) {$\circlearrowleft$};
\node (a1) at (-1,0) {$0$};
\node (a0) at (-1,1) {$0$};
\node (b0) at (0,1) {$F_S$};
\node (b1) at (0,0) {$F_S$};
\node (c1) at (1.5,0) {$\Lambda_S$};
\node (c0) at (1.5,1) {$X $};
\node (d1) at (3,0) {$G$};
\node (d0) at (3,1) {$G'$};
\node (e1) at (4,0) {$0$};
\node (e0) at (4,1) {$0$};
\draw[->] (a1) edge (b1) (b1) edge (c1)(c1) edge (d1)(d1) edge (e1);
\draw[->] (a0) edge (b0) (b0) edge (c0)(c0) edge (d0)(d0) edge (e0);
\draw[double equal sign distance] (b0) -- (b1);
\draw[right hook->](d0) -- (d1);
\draw[->] (c0) -- (c1);
\end{tikzpicture}
\end{center}
The upper short exact sequence splits if and only if $X$ is included in the subgroup $F_S \times \Tor \Lambda_S$ of $\Lambda_S$.
Indeed, if the upper sequence splits then $ F_S \times G' \simeq X \subseteq \Lambda_S$ and $G'$ is a torsion group, hence included in $\Tor \Lambda_S$.
Viceversa, if $X \subseteq F_S \times \Tor \Lambda_S$ then the projection onto the first factor gives a retraction of $F_S \hookrightarrow X$ and so the sequence splits.

Hence, the maximal subgroup $H_S(x)$ is isomorphic to the torsion subgroup of $\Lambda_S$, that is, $\Rad_{\Lambda_x} \Gamma_S / \Rad_{\Gamma} \Gamma_S$.
The obvious equality:
\[\left| \faktor{\Rad_{\Lambda_x} \Gamma_S}{\Rad_{\Gamma} \Gamma_S} \right|\cdot \left|\faktor{\Rad_{\Gamma} \Gamma_S}{\Gamma_S}\right| = \left|\faktor{\Rad_{\Lambda_x} \Gamma_S}{\Gamma_S} \right| \]
implies the equality $|H_S(x)|m_U (S)=m_{T_x}(S)$, for $T_x$ the covering of $U$ induced by the inclusion of the character groups $\Gamma \hookrightarrow \Lambda_x$.
Since $x$ is a coherent element, we obtain the equality $m_{T_x}=m_T$.
If $x$ and $x'$ in $\mathcal{C}$ are in the same orbit, then there exist two commutative diagrams
\begin{center}
\begin{tikzpicture}[scale=1.6]
\node (a) at (0,0) {$\Gamma$};
\node (b) at (1,0.5) {$\Lambda_x$};
\node (c) at (1,-0.5) {$\Lambda_{x'}$};
\node (a1) at (3.5,-0.5) {$U$};
\node (b1) at (3,0.5) {$T_{x'}$};
\node (c1) at (4,0.5) {$T_{x}$};
\draw[right hook->] (a) -- (b);
\draw[right hook->] (a) -- (c);
\draw[->] (b) -- (c);
\draw[->>] (b1) -- (a1);
\draw[->>] (c1) -- (a1);
\draw[->] (b1) -- (c1);
\end{tikzpicture}
\end{center}
This implies that $x$ and $x'$ induce equivalent extensions and we have a well defined map $\mathcal{C} / \Aut(G) \rightarrow \{T' \rightarrow U \mid m_{T'}=m_T \}$.

The surjectivity of the maps follows by taking a covering $T' \rightarrow U$ and the corresponding extension $\Gamma \hookrightarrow \Lambda'$.
Our hypothesis $m_{T'}=m_T$ implies that the cokernel is isomorphic to $G$, so $\Lambda'$ is represented by a coherent element of $\Ext^1(G,\Gamma)$.

The injectivity follows from the correspondence of \cref{lemma_rivest_tori}.
\end{proof}

\begin{cor} \label{cor_invariante_centrato}
For any centred toric arrangement, the data of the poset $\S$ together with $x \in \mathcal{C} / \Aut(G)$ are a complete invariant system for the arrangement up to automorphism of the torus.
\end{cor}

\begin{example}

We continue \Cref{ex:running}.

The group $G$ associated to the arithmetic matroid is the cokernel of the matrix $\left( \begin{smallmatrix}
1 & 0 \\
0 & 5
\end{smallmatrix} \right)$, hence $G \simeq \Z/5\Z$.
The matrix $\left( \begin{array}{c|c}
\id_2 & A
\end{array} \right)$ describes the three vectors $e_1, e_2, \frac{1}{2}(5e_1+3e_2)$ in the vector space $\Q^2$; let $\Gamma$ be the lattice in $\Q^2$ generated by these three vectors.

The surjective arithmetic matroid associated to $([3],\rk,m)$ is the arithmetic matroid that represents the vectors $e_1, e_2, \frac{1}{2}(5e_1+3e_2)$ in $\Gamma$.
The multiplicity function $m_U$ is defined by:
\begin{align*}
& m_U(\emptyset)=1\\
& m_U(e)=1  \quad \textnormal{for } e=1,2,3 \\
& m_U(1,2)= 2 \\
& m_U(1,3)= 3 \\
& m_U(2,3)= 5 \\
& m_U(1,2,3)=1
\end{align*}

The group $\Ext^1(G,\Gamma)$ is isomorphic to $(\Z/5\Z)^2$.
We look for the elements $x \in \Ext^1(G,\Gamma)$ which are coherent with the arithmetic matroid.
Notice that the subgroups of $G$ are only $0$ and $G$ (this is not true in general) and imposing the coherence conditions for $x$ yields:
\begin{gather*}
|H_\emptyset(x)|=\frac{m(\emptyset)}{m_U(\emptyset)}= 1 \Leftrightarrow H_\emptyset(x)=0 \Leftrightarrow x \neq 0 \in \Ext^1(G,\Gamma)\\
|H_i(x)|=\frac{m(i)}{m_U(i)}= 1 \Leftrightarrow H_i(x)=0 \Leftrightarrow \pi_i(x) \neq 0
\end{gather*}
The last implication holds for $i=1,2,3$.
If we identify $x\in \Ext^1(G,\Gamma)$ with pairs $(a,b)$ such that $a,b \in \Z/5\Z$ the conditions become, respectively:
\[ \pi_1(x)= b \neq 0 \qquad \pi_2(x)=a \neq 0 \qquad \pi_3(x)=3a-5b \neq 0\]
For the remaining subsets $S\subset [3]$, which are all of rank two, we have that $\rad_\Gamma \Gamma_S$ coincides with the whole lattice $\Gamma$.
In particular $\pi_S(x)=0$ and $H_S(x)=G$, therefore these conditions of coherence are always satisfied.

Summing all up, the coherent elements are $\mathcal{C}=\{(a,b)| a,b \neq 0\}$ (notice that the conditions $a \neq 0$ and $3a-5b \neq 0$ are equivalent) and the set $\mathcal{C} / \Aut(G)$ coincides with $\{a \in G | a \neq 0\}$ (take a representative with $b=3$).

We are going to built a representation of the arithmetic matroid for each element of $\mathcal{C} / \Aut(G)$, such that any two of them are non-isomorphic.
Let $\Lambda_x$ be the lattice in $\Q^2$ generated by $e_1, e_2, \frac{1}{2}(5e_1+3e_2), \frac{1}{5} (a e_1 + 3e_2)$ (recall that $x=(a,3)$).
A basis for $\Lambda_x$ is given by
\[\Big\{ e_1, \frac{1}{10}((2a-5)e_1+e_2) \Big\}\]
where $a \in \{1, \dots, 4\}$.
The three vectors $e_1, e_2, \frac{1}{2}(5e_1+3e_2)$ have coordinates in the basis above given by the matrix:
\[ C_a=\left( \begin{array}{ccc}
1 & 2a-5 & 3a-5 \\
0 & 10 & 15
\end{array} \right)\]
These are all the representations of the initial arithmetic matroid, up to left action of $\gl(2;\Z)$ and sign reverse of the columns.

In fact, the initial representation is equivalent to $C_2$:
\[ \left( \begin{array}{ccc}
-2 & -32 & -43 \\
1 & 21 & 29
\end{array} \right)=\left( \begin{array}{cc}
2 & -3  \\
-1 & 2
\end{array} \right)\left( \begin{array}{ccc}
1 & -1 & 1 \\
0 & 10 & 15
\end{array} \right)
\left( \begin{array}{ccc}
-1 &  &  \\
 & 1 & \\
& & 1
\end{array} \right)\]
\end{example}

\section{Cohomology algebra with integer coefficients}
\label{sect_cohom_int_coeff}
In this section we introduce the bigraded algebra $GOS^{\bullet,\bullet }(M(\Delta);\Z)$ of a toric arrangement $\Delta$, which turns out to be isomorphic to $\gr H^\bullet(M(\Delta);\Z)$, studied in \cref{sect_OS_algebra}.
In order to define this new algebra, we will use the data given by the poset of layers $\S$ and the coherent element $x\in \mathcal{C}/\Aut (G)$ associated to $\Delta$.
We will find that $GOS^{\bullet,\bullet }(M(\Delta);\Z)$ is an analogue (with integer coefficients) of the algebra $GOS^{\bullet,\bullet }(M(\Delta);\Q)$ we have seen in \cref{sect_Leray_ss}: in fact, $GOS^{\bullet,\bullet }(M(\Delta);\Z) \otimes_\Z \Q \simeq GOS^{\bullet,\bullet }(M(\Delta);\Q)$.

Let $\Delta=\{(a_e,\chi_e)\}_{e \in E}$ be a toric arrangement in normal form in the torus $T$ and $\S$ be its poset of layers.
Consider the character group $\Lambda$ of $T$ and its subgroup $\Gamma$ generated by $\{ \chi_e \}_{e\in E}$;
the inclusion $\Gamma \hookrightarrow \Lambda$ corresponds to a covering $T \rightarrow U$, where $U$ is the torus associated to $\Gamma$.
The group of covering automorphisms of $T\rightarrow U$ is $G=\Lambda/ \Gamma$; the hypertori $H_e=\mathcal{V}_T(1-a_e\chi_e)$ are stable under the action of $G$ and therefore each of them defines a hypertorus $H_e'=\mathcal{V}_U(1-a_e\chi_e) \subset U$.
Call $\Delta_U=\{(a_e, \chi_e)\}_{e \in E}$ the surjective toric arrangement given by $H'_e$ in $U$.

The poset $\S$ allows us to determine a lattice $\Gamma'$ and vectors $\{v_e\}_{e \in E} \subset \Gamma'$ such that $\Gamma \simeq \Gamma'$ through an isomorphism mapping each $v_e$ to $\chi_e$.
Since $\S$ determines the isomorphism class of the group $G$ (see, \Cref{remark_on_G}) and hence the conditions of coherence, the set $\mathcal{C}/ \Aut(G)$ can be recovered from $\S$.

Suppose now that we know the coherent element $x \in \mathcal{C} / \Aut(G)$ corresponding to $\Delta$.
From $x$ we build an extension $\Lambda'$ of $\Gamma'$ such that $\Lambda \simeq \Lambda'$ through an isomorphism that restricts to the previous one $\Gamma \simeq \Gamma'$.
With a slight abuse of notation, the lattices $\Lambda'$ and $\Gamma'$ will be called $\Lambda$ and $\Gamma$ as well.

Define the free $\Z$-module $V_\Z$ as the module $\bigwedge^\bullet \Lambda [ \{ y_{W,S} \}_{W,S}]$ with $S$ circuits associated to $W$.
This module is the direct sum of some copies of the external algebra $\bigwedge^\bullet \Lambda$ and can be endowed with a product $\cdot$ that coincides with the wedge product on the external algebra $\bigwedge ^\bullet \Lambda$ and is defined on the generators by
\[ y_{W,S} \cdot y_{W',S'}= 0 \]
if the tuple $SS'$ is dependent, or by
\[ y_{W,S} \cdot y_{W',S'} = \sum_{W''} y_{W'',SS'} \]
otherwise, where $W''$ runs over all connected components of $W \cap W'$.
Let the elements of $\Lambda$ be homogeneous of degree $(1,0)$ and $y_{W,S}$ be of degree $(0,|S|)$, so that $V_\Z$ is a bigraded $\Z$-algebra.

Define the map $\tilde{\psi}: V_\Z \rightarrow A^{\bullet, \bullet}(\Delta)$ (recall \cref{prod}) of bigraded algebras that sends $\lambda \in \Lambda$ to $\lambda^*(\omega) \otimes 1$ in the direct summand $H^1(T;\Z) \otimes H^0(\C^r;\Z)$ of $A^{\bullet, \bullet}(\Delta)$ and $y_{W,S}$ to $1 \otimes f_S$, where $f_S$ is the generator of the Orlik-Solomon algebra isomorphic to $H^{\rk W} (M(\Delta[W]);\Z)$.

Let $J_\Z$ be the ideal generated by the three relations:
\begin{enumerate}[1')]
\item $y_{W,S}=(-1)^{\sgn \sigma}y_{W,\sigma S}$ for all permutation $\sigma \in \mathfrak{S}_{|S|}$. \label{eq_reorder_Z}
\item $\lambda \cdot y_{W,S}=0$ for all $\lambda \in \Rad_\Lambda \Gamma_S$. \label{eq_ovvia_Z}
\item $\sum_{i=0}^{\rk W} (-1)^i y_{W,S\setminus s_i}=0$ for all circuits $S=\{s_0, \cdots, s_{\rk W}\}$ associated to $W$. \label{eq_relazioni_int}
\end{enumerate}

\begin{de} \label{def_GOS_int}
The graded toric Orlik-Solomon of the toric arrangement $\Delta$ with integer coefficients is $GOS^{\bullet,\bullet}(M(\Delta);\Z)$ defined as the quotient of $V_\Z$ by the ideal $J_\Z$.
\end{de}

\begin{lemma}
The algebra $GOS^{\bullet,\bullet}(M(\Delta);\Z)$ is a torsion free $\Z$-module.
\end{lemma}

\begin{proof}
The elements $y_{W,S}$ with $S$ no broken circuit associated to $W$ generate $GOS^{\bullet,\bullet}(M(\Delta);\Z)$ as $\bigwedge^\bullet \Lambda$-algebra.
Consequently, the module $GOS^{\bullet,\bullet}$ is the quotient of $\bigwedge^\bullet \Lambda [ \{ y_{W,S} \}_{(W,S) \in \mathcal{N}}]$ by the submodule $N$ generated by the elements $\lambda \cdot y_{W,S}$ for all $S$ associated to $W$ and $\lambda \in \Rad_\Lambda \Gamma_S$ (\ref{eq_ovvia_Z}).
The module $N$ has the property that for all integer $m$, if $m x$ belongs to $N$, then $x\in N$.
Thus $GOS^{\bullet,\bullet}(M(\Delta);\Z)$ is a torsion-free $\Z$-module.
\end{proof}
The map $\tilde{\psi}$ passes to the quotient by $J_\Z$, in fact an easy check shows that $J_\Z$ is contained in $\ker \tilde{\psi}$.
Call $\psi$ the induced map $GOS^{\bullet ,\bullet}(M(\Delta);\Z) \rightarrow A^{\bullet, \bullet}$.

Recall \cref{def_GOS} of $\mathcal{B}_W$ for $W \in \S$.
\begin{lemma} \label{lemma_generetors_int}
For any toric arrangement $\Delta$, the elements $(\prod_{\lambda \in S'} \lambda) \cdot y_{W,S}$, with $S$ no broken circuits associated to $W$ and $S'$ a subset of $\mathcal{B}_W$, generate $GOS^{\bullet,\bullet}(M(\Delta);\Z)$ as $\Z$-module.
\end{lemma}

The proof of \cref{lemma_generetors_int} is analogous to the proof of \cref{lemma_generators}. 

\begin{teo} \label{main_on_integers}
For every toric arrangement $\Delta$ in normal form, the poset of layers $\S$ and an element $x \in \mathcal{C} / \Aut(G)$ determine the graded algebra $GOS^{\bullet ,\bullet}(M(\Delta);\Z)$.

Moreover, the map $\psi \circ f$ is an isomorphism between $GOS^{\bullet ,\bullet}(M(\Delta);\Z)$ and $\gr H^\bullet(M(\Delta);\Z)$.
\end{teo}

\begin{proof}
The characters $\chi_i \in \Lambda$ for $i=1, \dots, n$ are determined, up to isomorphism, by $\S$ and $x$ (\cref{cor_invariante_centrato}).
The incidence relations among the hypertori of $\Delta$ are codified in $\S$.
Since the construction of $GOS^{\bullet ,\bullet}(M(\Delta);\Z)$ depends only on the characters and the incidence relations between hypertori, the first assertion follows.

We prove the surjectivity of $\psi$ by showing the surjectivity of $\tilde{\psi}$.
The image of $\tilde{\psi}$ contains $A^{1,0}$ and $A^{0,\bullet}$, because $\Lambda \simeq H^1(T;\Z)$ and the elements $f_S$ generate $H^{\rk W}(M(\Delta[W]);\Z)$.
It is obvious that $A^{p,q}= (A^{1,0})^p \cdot A^{0,q}$.

The map $\psi$ is injective since the two $\Z$-modules have the same rank equal to $ \sum_{q=0}^{r} 2^{r-q}|\mathcal{N}_q|$ and $GOS^{\bullet ,\bullet}(M(\Delta);\Z)$ is torsion-free.
Therefore, $\psi$ is an isomorphism.

The composition of $\psi$ with the isomorphism $f$ of \cref{main_on_cohomology} gives the desired isomorphism $GOS^{\bullet ,\bullet}(M(\Delta);\Z) \rightarrow \gr H^\bullet(M(\Delta);\Z)$.
\end{proof}

\section{An example}
\label{sect_example}
In this section we give an example of two toric arrangements with the same poset of intersections $\mathscr{S}$ and different invariant in $\mathcal{C} / \Aut(G)$.
By the theory we have developed so far, these two arrangements will have isomorphic graded cohomology ring $\gr^\bullet H^\bullet(M(\Delta);\Q)$ with rational coefficients and isomorphic cohomology modules with integer coefficients.
Nevertheless, their cohomology rings with integer coefficients are not naturally isomorphic, \ie there is no isomorphism between the cohomology rings that preserve the cohomology of the tori.

\begin{example}
\label{controesempio}
Let $\Delta$ and $\Delta'$ be the centred arrangements in $(\C^*)^2$ given by the three characters in $\Lambda=\Z^2$ and, respectively, in $\Lambda'=\Z^2$:
\begin{align*}
\left( \begin{array}{c|c|c}
\chi_1 & \chi_2 & \chi_3
\end{array} \right)
 = \left( \begin{array}{ccc}
1 & 1 & 2 \\ 
0 & 7 & 7
\end{array} \right) \\
\left( \begin{array}{c|c|c}
\chi_1' & \chi_2' & \chi_3'
\end{array} \right)
 = \left( \begin{array}{ccc}
1 & 2 & 3 \\ 
0 & 7 & 7
\end{array} \right)
\end{align*}
\end{example}
These two arrangements have the same poset of intersections $\S$ and the same arithmetic matroid $(E,\rk, m)$.
The first two characters of $\Delta$ (respectively, of $\Delta'$) form a basis of the arithmetic matroid (or, equivalently, of the vector space $\Lambda \otimes \Q$).
Following the ideas in \cref{sect_On uniqueness of representability}, we write the three characters in this chosen basis; the coordinates of $\chi_1, \chi_2, \chi_3$ (respectively of $\chi'_1, \chi'_2, \chi'_3$) are
\[ D=\left( \begin{array}{ccc}
1 & 0 & 1 \\
0 & 1 & 1
\end{array} \right)\]
in both cases. Hence the matrix $A= \left( \begin{smallmatrix}
1 \\ 1
\end{smallmatrix} \right)$ is in normal form with respect the unique maximal tree in the associated graph.
Let $\Gamma$ be the lattice $\Z^2$, $U$ the two dimensional torus and $\Delta_U$ the arrangement in $U$ described by the matrix $D$.

The inclusion $\Gamma \hookrightarrow \Lambda$ related to the arrangement $\Delta$, defined by
\begin{align*}
 e_1 & \longmapsto \chi_1 \\
 e_2 & \longmapsto \chi_2 \\
 e_1+e_2 & \longmapsto \chi_3 
\end{align*}
has cokernel $\Lambda / \Gamma \simeq \Z_7$\footnote{We denote with $\Z_n$ the cyclic group on $n$ elements.} and the same holds for the inclusion $\Gamma \hookrightarrow \Lambda'$ given by $\Delta'$.
We want to find an element in $ \Ext^1(\Z_7,\Gamma ) / \Z_7^*$ that represents the extension $\Gamma \hookrightarrow \Lambda$.
The group $\Ext^1(\Z_7, \Gamma )$ is isomorphic to $\Z_7 \times \Z_7$ and the elements of $\Ext^1(\Z_7 , \Gamma)$ coherent with $\Delta_U$ and $m$ are
\[ \mathcal{C}=\{ (a,b) \in \Z_7 \times \Z_7 \mid a,b,a+b \neq 0 \textnormal{ in } \Z_7\} \]
hence the quotient by the natural action of $\Z_7^*$ is
\[ \faktor{\mathcal{C}}{\Z_7^*} \simeq \{ x \in \Z_7 \mid x,x+1 \neq 0 \textnormal{ in } \Z_7\}, \]
where each element $x$ corresponds to the $\Z_7$ orbit of $(x,1)$ in $\mathcal{C}$.
The coherent element associated to $\Delta$ is $x=1$ and the one associated to $\Delta'$ is $x'=2$.

The two algebras $GOS^{\bullet,\bullet}(M(\Delta);\Q)$ and $GOS^{\bullet,\bullet}(M(\Delta');\Q)$ are isomorphic because, as the construction in \cref{sect_OS_algebra}, they depend only on the poset $\S$ (or on $(E,\rk,m)$ for centred toric arrangements).

\begin{prop}
There is no isomorphism of algebras $f$ between the cohomologies with integer coefficients of $M(\Delta)$ and $M(\Delta')$ of \cref{controesempio} such that the following diagram commutes
\begin{center}
\begin{tikzpicture}[scale=1.8]
\node (xa1) at (1,1) {$H^\bullet(M(\Delta);\Z)$};
\node (x1) at (-1,0.5) {$H^\bullet(T;\Z)$};
\node (x2) at (1,0) {$H^\bullet(M(\Delta');\Z)$};
\draw[right hook->] (x1) to (xa1);;
\draw[right hook->] (x1) to (x2);
\draw[->,dashed] (xa1) to node[left]{$f$} (x2);
\end{tikzpicture}
\end{center}
\end{prop}

\begin{proof}
Suppose that such an isomorphism exists, then it induces a graded isomorphism between the algebras $\gr H^\bullet(M(\Delta);\Z)$ and $\gr H^\bullet(M(\Delta);\Z)$.
By \cref{main_on_integers} the algebras $GOS^{\bullet,\bullet}(M(\Delta);\Z)$ and $GOS^{\bullet,\bullet}(M(\Delta');\Z)$ are isomorphic as bigraded algebras.
We want to study the behavior of the product of two homogeneous elements of degree $(1,0)$ and $(0,1)$ respectively.
Consider $\lambda \in \Lambda$ and $\sum_{i=1}^3 a_i y_{H_i,i}$ with $a_i \in \Z$; then
\[ \lambda \cdot \sum_{i=1}^3 a_i y_{H_i,i} =0 \Longleftrightarrow \forall\, i=1,2,3 \quad a_i=0 \textnormal{ or } \lambda \in \Z \chi_i \]
Therefore, the set
\[ \{\lambda \in \Lambda \mid GOS^{0,1}(M(\Delta);\Z) \xrightarrow{\cdot \lambda} GOS^{1,1}(M(\Delta);\Z) \textnormal{ is not injective}\}\]
coincides with the union of the three subgroups $\Z \chi_i$, for $i=1, \dots, 3$.
Analogously, we can define intrinsically the three subgroups $\Z \chi'_i$, for $i=1, \dots, 3$, in $\Lambda'$.
We claim that there is no isomorphism $\Lambda \xrightarrow{\sim} \Lambda'$ that sends each subgroup $\Z \chi_i$ into the subgroup $\Z \chi'_{\sigma(i)}$, for some $\sigma \in \mathfrak{S}_3$.

In $GOS^{\bullet,\bullet}(M(\Delta);\Z)$ there exist two elements that generate two different subgroups $\Z \chi_i$ whose sum is in the sublattice $7H^1(T;\Z)$ (namely $\left( \begin{smallmatrix}
-1 \\ 
0
\end{smallmatrix} \right) + \left( \begin{smallmatrix}
1 \\ 
7
\end{smallmatrix} \right)= 7\left( \begin{smallmatrix}
0 \\ 
1
\end{smallmatrix} \right)$), while this property does not hold for $\Lambda'$ and the three subgroups  $\Z \chi_i'$.
We have thus proven that the cohomology rings $H^\bullet(M(\Delta);\Z)$ and $H^\bullet(M(\Delta');\Z)$ are not naturally isomorphic.
\end{proof}

\section{Discriminantal toric arrangements}
\label{sect_discr}

We want to study the continuous deformations of a toric arrangement.
Since the characters group $\Lambda$ of a toric arrangement is a discrete set, no deformation can change the set of characters in $\Delta$, however it is possible that some hypertorus in the arrangement are translated.

A particular nice type of deformation is the \textit{poset isotopy}, which is by definition a deformation  that does not change the poset of layers.
Two toric arrangements are said to be \textit{poset isotopy equivalent} if there exists a poset isotopy that deforms one into the other.
This notion has already been defined in the context of hyperplane arrangements, see \cite[Definition 5.27]{orlik-terao}.

There are two natural questions: 
\begin{enumerate}
\item Are two poset isotopy equivalent toric arrangements diffeomorphic? Are they homeomorphic? Are they homotopy equivalent?
\item Are two toric arrangements with isomorphic poset of layers and same characters poset isotopy equivalent?
\end{enumerate}
In this section we will give a negative answer to the second question and a positive one to the first question.

The following example was suggested by Filippo Callegaro and it is a counterexample to the second question.

\begin{example} \label{es_homotopy}
Let $\Delta$ and $\Delta'$ be the following two toric arrangements in $T^2=\hom (\Z^2, \C^*)$:
\begin{align*}
\Delta &= \left\{(1,\left( \begin{matrix}
1 \\ 0
\end{matrix} \right)), (1,\left( \begin{matrix}
1 \\ 7
\end{matrix} \right)), (1,\left( \begin{matrix}
0 \\ 1
\end{matrix} \right)), (\zeta,\left( \begin{matrix}
0 \\ 1
\end{matrix} \right))
\right\} \\
\Delta' &= \left\{(1,\left( \begin{matrix}
1 \\ 0
\end{matrix} \right)), (1,\left( \begin{matrix}
1 \\ 7
\end{matrix} \right)), (1,\left( \begin{matrix}
0 \\ 1
\end{matrix} \right)), (\zeta^2,\left( \begin{matrix}
0 \\ 1
\end{matrix} \right))
\right\}
\end{align*}
where $\zeta$ is a primitive $7^{th}$-root of unity.

\begin{figure}
    \centering
    \begin{subfigure}[b]{0.48\textwidth}
		\begin{tikzpicture}[scale=0.8]
			\draw[gray] (0, 0) rectangle (7,7);
			\draw [line width=1pt] (0,3.5) -- (7,3.5);
			\foreach \x in {1,2,...,7}{
				\draw [line width=1pt] (\x-1,0) -- (\x,7);
				\draw (\x-0.5,3.5) node [circle, draw, fill=black, inner sep=0pt, minimum width=4pt] {};
			}
			\draw [line width=1pt](0.5,0) -- (0.5,7);
			\draw [color=red, line width=1pt](1.5,0) -- (1.5,7);
		\end{tikzpicture}
		\caption{$M(\Delta)$}
	\end{subfigure}
    \begin{subfigure}[b]{0.48\textwidth}
		\begin{tikzpicture}[scale=0.8]
			\draw[gray] (0, 0) rectangle (7,7);
			\draw [line width=1pt] (0,3.5) -- (7,3.5);
			\foreach \x in {1,2,...,7}{
				\draw [line width=1pt] (\x-1,0) -- (\x,7);
				\draw (\x-0.5,3.5) node [circle, draw, fill=black, inner sep=0pt, minimum width=4pt] {};
			}
			\draw [line width=1pt](0.5,0) -- (0.5,7);
			\draw [color=red, line width=1pt](2.5,0) -- (2.5,7);
		\end{tikzpicture}
		\caption{$M(\Delta')$}
	\end{subfigure}
	\caption{Representations of the two arrangements on the compact subtorus. The red subtori are the hypertori $H_4$ and $H'_4$.}
\end{figure}
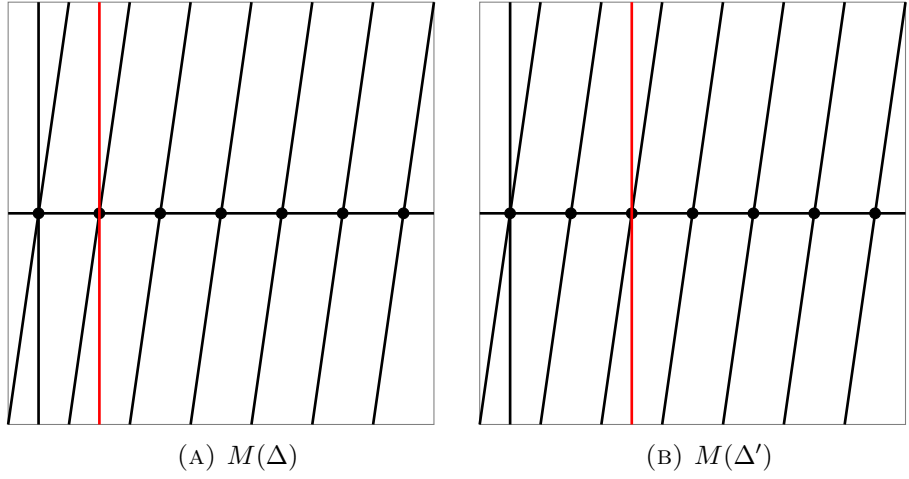

These arrangements have the same poset of intersections $\mathscr{S}$ and the same characters $\chi_i$ for $i=1, \dots, 4$, therefore by \cref{main_on_integers} their bigraded cohomology groups are isomorphic.
\end{example}

\begin{de}
A layer $W \in \mathscr{S}$ is \textit{generic} if there exist exactly $\rk W$ hypertori containing $W$.

A poset is \textit{generic} if all layers in $\mathscr{S}$ are generic.

A poset $\mathscr{S}$ is \textit{nearly generic} if there exists a layer $\overline{W} \in \mathscr{S}$ such that all layers $W$ not containing $\overline{W}$ are generic.
\end{de}
Clearly, generic arrangements are nearly generic.
We fix $n$ characters $\chi_i$ in $\Lambda$ and study the poset of layers of $\Delta=\{(a_i, \chi_i)\}_{i=1,\dots, n}$, where $a_i$ ranges in $\C^*$.
From now on, we suppose all characters $\chi_i$, $i=1,\dots, n$, primitive and we consider the natural surjective map from $E=\{1,\dots, n\}$ onto $\mathscr{S}_1(\{(a_i, \chi_i)\}_{i=1,\dots, n})$, the set of atoms.
We will understand that all isomorphisms between these posets will be over $E$.

We adopt the convention that a list of indexes $(j_1, \cdots, j_k)$ is denoted by an underlined letter $\underline{j}$.

\begin{teo} \label{deform}
The subset of $(\C^*)^n$ given by
\[ L(\mathscr{S})= \{\underline{a}= (a_i)_i \in (\C^*)^n \mid \mathscr{S}(\{(a_i, \chi_i)\}_{i=1,\dots, n}) \simeq \mathscr{S} \} \]
is a smooth locally closed subset of $(\C^*)^n$.
Moreover, if $\mathscr{S}$ is nearly generic then $L(\mathscr{S})$ is connected or empty.
\end{teo}

For each $\underline{a} \in (\C^*)^n$ define the hypertori $H_i=H_i(\underline{a})= \mathcal{V}_T (1-a_i \chi_i)\subset T$ for $i=1,\dots , n$.

\begin{lemma} \label{lemma_tori_discr}
For all sublists $\underline{j}$ of $(1, \dots, n)$, the subset 
\[B_{\underline{j}}= \{ \underline{a}\in (\C^*)^n \mid H_{j_1}(\underline{a}) \cap \cdots \cap H_{j_k}(\underline{a})\neq \emptyset\}\]
is a connected torus in $(\C^*)^n$.
Moreover, the intersection of $H_{j_1}, \cdots, H_{j_k}$ in $T$ is independent, up to translation, of the point $\underline{a} \in B_{\underline{j}}$.
\end{lemma}

\begin{proof}
Without loss of generality we suppose $j_i=i$, for $i=1, \dots, k$ and study the subtorus $Y$ of 
\[(\C^*)^{n+r}=\spec \C[ a_1^{\pm 1}, \cdots, a_n^{\pm 1}, z_1^{\pm 1}, \cdots, z_r^{\pm 1}]\] given by the equations $I= (1-a_i\chi_i)_{i=1, \dots, k}$.
The rings morphism
\[\C[ a_1^{\pm 1}, \cdots, a_n^{\pm 1}] \longrightarrow \faktor{ \C[ a_1^{\pm 1}, \cdots, a_n^{\pm 1}, z_1^{\pm 1}, \cdots, z_r^{\pm 1}]}{I}\]
induces a projection between the associated tori:
\[\begin{matrix} p \colon & Y & \longrightarrow & (\C^*)^n \\ & (\underline{a}, \underline{z}) & \longmapsto & \underline{a} & \end{matrix}\]
The image of the map  $p$ is the closed subset $B$ described by the contracted ideal 
\[I^c = (1-a_i\chi_i)_{i=1,\dots, k} \cap \C[a_1^{\pm 1}, \cdots, a_n^{\pm 1}]\]
Given that the intersection of $H_1(\underline{a}), \cdots, H_k (\underline{a})$ is nonempty if and only if $\underline{a}$ is in the image of $p$, $B$ coincides with $B_{\underline{j}}$.
The elimination ideal of a binomial ideal is still binomial: this is a standard fact about binomial ideals and Gr\"obner bases (for a proof see \cite[Corollary 1.3]{eisenbud1996}).
In our case, since $I$ is a binomial ideal, $I^c$ is binomial and the closed subset $B$ is a torus.
Moreover, $B$ is connected since it is the image of the connected torus $Y$ under the map $p$.

The fibers of the map $p$ are either empty or a torus of codimension $\rk \left( \begin{array}{ccc}
\chi_1 & \cdots & \chi_k
\end{array} \right)$
and with $m$ connected components, where $m$ is the $k^{th}$-deter\-mi\-nant divisor of the matrix $\left( \begin{array}{ccc}
\chi_1 & \cdots & \chi_k
\end{array} \right)$.
The fiber of a point in $B$, seen as torus in $T$, is obtained from  any other nonempty fiber by translations.
\end{proof}

Recall that $\chi_i$, $i=1,\dots,n$, are characters of a $r$-dimensional torus.

\begin{remark}
Let $\underline{j}$ be a sublist of $(1,\dots, n)$ of cardinality $k$.
The torus $Y$ is of dimension $n+r-k$, so the set $B_{\underline{j}}$ is of dimension $n+r-k-\rk \underline{j}$.
In particular, $B_{\underline{j}}$ is a hypertorus in $(\C^*)^n$ if and only if $\underline{j}$ is a circuit (\ie $\rk \underline{j} = r-k+1$).
\end{remark}

\begin{de}\label{def_discriminantal_TA}
The centred toric arrangement $D(\chi_1, \chi_2, \dots, \chi_n)$ given by the sets $B_{\underline{j}}$,for all circuits $\underline{j}$, in $T'=(\C^*)^n$ is called \textit{discriminantal toric arrangement} associated to the $n$ characters $\chi_i$, $i=1,\dots , n$.
\end{de}

\begin{proof}[Proof of \Cref{deform}]
If $\S$ is not $\mathscr{S}(\{(b_i, \chi_i)\}_{i=1,\dots, n})$ for some $b_i$, there is nothing to prove.
Otherwise, for each layer $W$ of $\S$, let $\underline{j}(W)$ be the ordered set:
\[\underline{j}(W) = \{ i \mid A_i \leq W \}\]
where $A_i$ is the atom of $\S$ associated to $i$.

The condition $\mathscr{S}(\{(a_i, \chi_i)\}_{i=1,\dots, n})= \mathscr{S}$ is equivalent to 
\[
\forall \, \underline{a} \in (\C^*)^n \, \exists \, W \in \S 
\, \left( \underline{a} \in B_{\underline{j}} \Leftrightarrow \underline{j}\subseteq \underline{j}(W) \right)
\]
By \Cref{lemma_tori_discr}, the set 
\[L(\mathscr{S})= \bigcap_{ W\in \mathscr{S_r}} B_{\underline{j}(W)} \setminus \bigcup_{\substack{ \underline{j} \not \subseteq \underline{j}(W) \\ \forall W\in \mathscr{S_r}}} B_{\underline{j}}\]
is locally closed in $(\C^*)^n$ and open in the torus $\bigcap_{ W\in \mathscr{S_r}} B_{\underline{j}(W)}$, hence it is also smooth.

If $W$ is a generic layer, then we have that $B_{\underline{j}}=(\C^*)^n$ for all sets $\underline{j}$ included in $\underline{j}(W)$.
Let $\mathscr{S}$ be a nearly generic poset and $\overline{W}$ non-generic maximal layer; then the equality $L(\mathscr{S})= B_{\underline{j}(\overline{W})} \setminus \bigcup_{\text{some } \underline{j}} B_{\underline{j}} $ holds.
If $L(\mathscr{S})$ is nonempty, it is an irreducible set; it is connected in the Zariski topology and thus also in the euclidean one.
\end{proof}

\begin{example}[continuation of \Cref{es_homotopy}]
Let $\mathscr{S}$ be the poset associated to $\Delta$ or, equivalently, to $\Delta'$.
The discriminantal toric arrangement associated to $\left( \begin{smallmatrix}
1 & 1 & 0 & 0 \\
0 & 7 & 1 & 1 
\end{smallmatrix} \right)$ is the centred toric arrangement in $(\C^*)^4= \spec \C[\{a_i^{\pm 1}\}_{i \leq 4}]$ given by the subtori
\begin{align*}
B_{3,4}&= \mathcal{V}(1-a_3a_4^{-1})\\
B_{1,2,3}&= \mathcal{V}(1-a_1a_2^{-1}a_3^7) \\
B_{1,2,4}&= \mathcal{V}(1-a_1a_2^{-1}a_4^7) \\
B_{1,2,3,4}&= \mathcal{V}(1-a_1a_2^{-1}a_3^7, 1-a_3a_4^{-1}).
\end{align*}
All the others $B_{\underline{j}}$ are equal to $(\C^*)^4$ and consequently the subset $L(\mathscr{S})$ is: 
\[L(\mathscr{S})=B_{1,2,3} \cap B_{1,2,4} \setminus B_{3,4}=\{\underline{a} \mid a_1^{-1}a_2=a_3^7=a_4^7, \, a_3 \neq a_4\}\]
Hence the set $L(\mathscr{S})$ is the disjoint union of six connected $2$-dimensional tori.
The two arrangements $\Delta$ and $\Delta'$ belong to different connected components so they cannot be deformed one into the other by means of translations.
We have thus shown that $\Delta$ and $\Delta'$ have the same characters and the same poset of intersections but are not poset isotopy equivalent, see \Cref{def:poset_isotopic_equivalent}.
\end{example}

\begin{de}\label{def:poset_isotopic_equivalent}
A \textit{deformation} of a toric arrangement is a collection of $n$ hypersurfaces $H_i$ in $(\C ^*)^r \times B$ (where $B$ is an algebraic variety over $\C$) such that for every point $b\in B$ the subset $H_i \cap \pr_2^{-1}(b)$ is a hypertorus in $(\C^*)^r \times \{b\}$.
We call $M_b$ the open set $(\C^*)^r \times \{b\} \setminus \bigcup_{i\leq n} H_i$.

A deformation is said to be a \textit{poset-preserve deformations} if the poset of layers of $M_b$ does not depend on the point $b \in B$.

We said that two toric arrangements $M_1, M_2 \subset T$ are \textit{poset isotopic equivalent} if there exists a layers-preserve deformation in $T \times B$, $B$ connected, such that the pair $(T,M_1)$ (and $(T,M_2)$) is isomorphic to a fiber $(T\times \{b_1\} ,M_{b_1})$ (and, respectively, to $(T\times \{b_2\} ,M_{b_2})$).
\end{de}

The next result is an analogous to the one on hyperplane arrangements of Randell \cite{Randell}.
\begin{teo} \label{teo_diffeom}
If the toric arrangements $M, M' \subset T$ are poset isotopic equivalent then $M$ and $M'$ are diffeomorphic.
\end{teo}

Since the group character $\Lambda$ of a torus $T$ is a discrete set, two poset isotopy equivalent arrangements $M$ and $M'$ are described by the same characters $\chi_1, \cdots, \chi_n$.
Let $\S$ be the poset of layers of $M$ or equivalently of $M'$, the base $B$ of the deformation can be chosen to be a connected component of $L(\S)$.
Call $U$ the closure of $B$ in $(\C^*)^n$, clearly $U$ is a connected torus (possibly translated) of dimension $m$.

Consider the torus $T \times U$ with coordinates $(z_i, a_j)$ as before, with hypertori defined by $1-a_j \chi_j(z)=0$ and call $\tilde{M}$ the toric arrangement $(T \times U) \setminus \bigcup_{i=1,\dots,n} \mathcal{V}(1-a_j \chi_j(z))$. 
Choose a component-wise embedding of $T \times U$ in $(\mathbb{P}^1)^r \times (\mathbb{P}^1)^m $.
The main result of \cite{wonderfulmodel} is that there exists a smooth project variety $\mathcal{W}$ obtained from $(\mathbb{P}^1)^{r+m}$ by means of a suitable sequence of blow-ups that contains $\tilde{M}$.
Call $\mathcal{W}_B$ the inverse image of $(\mathbb{P}^1)^{r}\times B$ through the natural map $\mathcal{W}\rightarrow (\mathbb{P}^1)^{r+m}$.

Let $w: \mathcal{W} _B \rightarrow B$ be the composition of the natural map $\mathcal{W}_B \rightarrow (\mathbb{P}^1)^r \times B$ with the projection onto the second component.

\begin{lemma} \label{lemma_smooth_proj}
The map $w: \mathcal{W} _B \rightarrow B$ is a projective smooth map.
\end{lemma}

\begin{proof}
The morphism $\mathcal{W}_B \rightarrow (\mathbb{P}^1)^r \times B$ is projective and smooth by base change of a blow-up map.
The projection $(\mathbb{P}^1)^r \times B \rightarrow B$ is projective and smooth.
So the composition $w$ is smooth and projective.
\end{proof}


\begin{lemma}\label{lemma_Whitney_strat}
The variety $\mathcal{W}_B$ admits a Whitney stratification whose open stratum is naturally isomorphic to $\tilde{M}$.
\end{lemma}

\begin{proof}
The complement $\mathcal{K}$ of $\tilde{M}$ in $\mathcal{W}_B$ is a union of some smooth divisors $\mathcal{K}_i$, $i=1, \dots,k$.
Moreover $\mathcal{K}$ is normal crossing (see \cite{wonderfulmodel}).
Consider the stratification given by the closed sets $\mathcal{K}_i$, $i=1, \dots, k$.
Each stratum has smooth closure in $\mathcal{W}_B$, therefore by \cite[Lemma]{Randell} this is a Whitney stratification. 
\end{proof}

To prove of \cref{teo_diffeom} we follow the ideas of \cite{Randell}.

\begin{proof}[Proof of \cref{teo_diffeom}]
We apply the Thom's isotopy theorem (\cite[Section I Theorem 1.5]{Goresky-MP}) to the map $w$.
Consider the Whitney stratification on $\mathcal{W}_B$ of \cref{lemma_Whitney_strat}.
Since the poset of layers is the same for every point $b\in B$, the restriction of the map $w: \mathcal{W} _B \rightarrow B$ to every stratum is a submersion. 
Hence for all $b \in B$, there exists a smooth stratum-preserving map $h$ such that the diagram below commutes:
\begin{center}
\begin{tikzpicture}[scale=1.8]
\node (xa1) at (1,1) {$B \times w^{-1}(b)$};
\node (x1) at (-1,1) {$\mathcal{W}_B$};
\node (x2) at (0,0) {$B$};
\draw[->] (x1) to node[above]{$h$} (xa1);
\draw[->] (x1) to node[left]{$w$} (x2);
\draw[->] (xa1) to node[right]{$pr_1$} (x2);
\end{tikzpicture}
\end{center}
Thus for all $b'\in B$, $M_b$ and $M_{b'}$ are diffeomorphic.
\end{proof}

\section{Cohomology generated in degree one}
\label{sect_gen_deg_one}
In this section we will analyze the property of the cohomology ring of being generated in degree one.
We will show that this property depends only on the poset of intersections and we will give a combinatorial criterion to determine when this property holds.

\begin{lemma}
Let $B$ be a graded algebra and $\{F_i\}_{i \in \N}$ an increasing homogeneous filtration such that $\bigcup_{i \in \N} F_i=B$.
Then $B$ is generated in degree one if and only if $\gr_{F} B$ does.
\end{lemma}

Since $H^\bullet(T;\Z)\simeq H^\bullet(\C^*;\Z)^{\otimes r}$ is generated in degree one, $H^\bullet(M;\Z)$ is generated in degree one if and only if $A^{0,\bullet}$ is as well.
By \cref{main_on_integers}, we work on $GOS^{0,\bullet}(M;\Z)$ instead of on $A^{0,\bullet}$.
A similar argument show that $H^\bullet(M;\Q)$ is generated in degree one if and only if the same holds for $GOS^{0,\bullet}(M;\Q)$.

In this section we will work with the assumption that $\Delta$ is a primitive arrangement.

\begin{prop}
Let $\Delta$ be a primitive toric arrangement and $S$ be a standard list of cardinality $k$.
The following formula holds in $GOS^{0,\bullet}(M;\Z)$:
\[ y_S= \sum_{\substack{W\in \mathscr{S}_k \\ W \in \bigcap S}} y_{W,S}\]
\end{prop}

\begin{proof}
The formula is an easy consequence of the definition of product (\cref{def_GOS_int}) in the algebra $GOS^{\bullet,\bullet}(M;\Z)$.
\end{proof}

\begin{remark}
By \cref{lemma_generetors_int}, the elements $y_{W,S}$ can be written uniquely as linear combinations of the $y_{W,S'}$ with $S'$ no broken circuit, using the formulae \ref{eq_reorder_Z}- \ref{eq_relazioni_int}.
\[ y_{W,S} = \sum_{T \textnormal{ no broken}} r_{W,T}^S y_{W,T}\]
The coefficients $r^S_{W,T}$ are uniquely determined by the poset $\mathscr{S}_{\leq W}$.
\end{remark}

We proved that if $y_{S} = \sum r^S_{W,T} y_{W,T}$, where the sum runs over all layers $W$ and all no broken circuits $T$ associated to $W$, then the coefficients $r^S_{W,T}$ are integers.
We collect this data in a matrix $R^k$ whose columns are given by $(r^S_{W,T})_{W,T}$ for all $S$ of cardinality $k$; these are merely the coordinates of the element $y_S$ with respect to the basis $\{y_{W,T} \mid T \textnormal{ no broken circuit} \}$.

\begin{teo}\label{teo_gen_deg_uno_rat}
The algebra $H^\bullet(M(\Delta);\Q)$ is generated in degree one if and only if all the matrices $\{R^k \}_ {k \leq r}$ have rank equal to $\dim_{\Q}H^k(M(\Delta);\Q)$.
\end{teo}

\begin{proof}
Fix the degree $k$.
The submodule $H^k(M;\Q)$ is included in the subalgebra generated by $H^1(M;\Q)$ if and only if $R^k$ has a right inverse; this happens if and only if $R^k$ has rank equal to $\dim_{\Q}H^k(M(\Delta);\Q)$.
\end{proof}

\begin{de}
The \textit{$r$-th determinant divisor} of an integer matrix is the greatest common divisor of all the determinants of its minors of size $r$.
\end{de}

We call $d_k=\rk_{\Z}H^k(M(\Delta);\Z)$.
\begin{teo}
The algebra $H^\bullet(M(\Delta);\Z)$ is generated in degree one if and only if all the matrices $\{R^k \}_{k\leq r}$ have $d_k$-th determinant divisor equal to one.
\end{teo}

\begin{proof}
As in the proof of \cref{teo_gen_deg_uno_rat}, the submodule $H^k(M;\Z)$ is contained in the subalgebra generated by $H^1(M;\Z)$ if and only if $R^k$ has a right inverse with integer coefficients.
By the Smith normal form, this right inverse exists if and only if the $d_k$-th determinant divisor is equal to one.
Notice that if $R^k$ has strictly less then $d_k$ columns, then the $d_k$-th determinant divisor is zero.
\end{proof}


\bibliography{Toric_arr.bib}{}
\bibliographystyle{acm}
\end{document}